\documentclass[a4paper]{amsart} 
\usepackage{graphicx}
\usepackage{color}
\usepackage{amssymb,amsmath} 
\usepackage[arrow,curve,matrix,tips]{xy}
\usepackage{enumerate}
\usepackage{esint}
\usepackage[arrow,curve,matrix,tips]{xy}
\usepackage{setspace}
\makeatletter
\def\@tocline#1#2#3#4#5#6#7{\relax
  \ifnum #1>\c@tocdepth 
  \else
    \par \addpenalty\@secpenalty\addvspace{#2}%
    \begingroup \hyphenpenalty\@M
    \@ifempty{#4}{%
      \@tempdima\csname r@tocindent\number#1\endcsname\relax
    }{%
      \@tempdima#4\relax
    }%
    \parindent\z@ \leftskip#3\relax \advance\leftskip\@tempdima\relax
    \rightskip\@pnumwidth plus4em \parfillskip-\@pnumwidth
    #5\leavevmode\hskip-\@tempdima
      \ifcase #1
       \or\or \hskip 1em \or \hskip 2em \else \hskip 3em \fi%
      #6\nobreak\relax
    \dotfill\hbox to\@pnumwidth{\@tocpagenum{#7}}\par
    \nobreak
    \endgroup
  \fi}
\makeatother

\def\N{{\mathbb{N}}}
 
\def\Q{{\mathbb{Q}}}
\def\Z{{\mathbb{Z}}} 
\def\theta{{\vartheta}} 
\def\phi{{\varphi}}
\def\epsilon{{\varepsilon}}

\mathchardef\ordinarycolon\mathcode`\: \mathcode`\:=\string"8000
\begingroup \catcode`\:=\active
\gdef:{\mathrel{\mathop\ordinarycolon}}
\endgroup

\newtheorem{theorem}{Theorem}[section]
\newtheorem{lemma}[theorem]{Lemma}

\newtheorem{proposition}[theorem]{Proposition}
\newtheorem{corollary}[theorem]{Corollary}
\newtheorem{mainlemma}[theorem]{Main Lemma}

\theoremstyle{definition} 
\newtheorem{definition}[theorem]{Definition}
\newtheorem{example}[theorem]{Example}

\newtheorem{notation}[theorem]{Notation}

\newtheorem{fact}[theorem]{Facts}

\newtheorem{remark}[theorem]{Remark}

\numberwithin{equation}{section} 
\makeatletter
\def\moverlay{\mathpalette\mov@rlay}
\def\mov@rlay#1#2{\leavevmode\vtop{%
   \baselineskip\z@skip \lineskiplimit-\maxdimen
   \ialign{\hfil$\m@th#1##$\hfil\cr#2\crcr}}}
\newcommand{\charfusion}[3][\mathord]{
    #1{\ifx#1\mathop\vphantom{#2}\fi
        \mathpalette\mov@rlay{#2\cr#3}
      }
    \ifx#1\mathop\expandafter\displaylimits\fi}
\makeatother

\begin{document} 
\title[A Baer-Krull Theorem for quasi-ordered rings]{Compatibility of quasi-orderings and valuations; A Baer-Krull Theorem for quasi-ordered Rings}

\author{Salma Kuhlmann, Simon M\"uller}

\begin{center} {\small{\today}} \end{center}

\begin{abstract}
In his work of 1969, Merle E. Manis introduced valuations on commutative rings. Recently, the class of totally quasi-ordered rings was developed in \cite{sim}. In the present paper, given a quasi-ordered ring $(R,\preceq)$ and a valuation $v$ on $R,$ we establish the notion of compatibility between $v$ and $\preceq,$ leading to a definition of the rank of $(R,\preceq).$ \vspace{1mm}

\noindent
Moreover, we prove a Baer-Krull Theorem for quasi-ordered rings: fixing a Manis valuation $v$ on $R,$ we characterize all $v$-compatible quasi-orders of $R$ by lifting the quasi-orders from the residue class domain to $R$ itself.
\end{abstract}

\maketitle

\section{Introduction}

\noindent
There have been several attempts to find a uniform approach to orders and valuations. In \cite{efrat} for instance, Ido Efrat simply defined localities on a field to be either orders or valuations. S.M. Fakhruddin introduced the notion of (totally) quasi-ordered fields $(K,\preceq)$ and proved the dichotomy, that any such field is either an ordered field or else there exists a valuation $v$ on $K$ such that $x \preceq y$ if and only if $v(y) \leq v(x)$ (\cite[Theorem 2.1]{fakh}). Thus, Fakhruddin found a way to treat these two classes simultaneously. Inspired by this result, the second author of this paper established the said dichotomy for commutative ring with $1$ (\cite[Theorem 4.6]{sim}). \vspace{1mm}

\noindent
The aim of the present paper is to continue our study of quasi-ordered rings. To this end we consider important results from real algebra, which are also meaningful if the order is replaced by a quasi-order. The paper is organized as follows:
\vspace{1mm}

\noindent
In section 2 we briefly recall ordered and valued rings, and give our definition of quasi-ordered rings (see Definition \ref{qoring}). Moreover, we quote the two theorems that we want to establish for this class (see Theorems \ref{compfield} and \ref{baerkrullfield}). \vspace{1mm}

\noindent
Section 3 deals with the notion of compatibility between quasi-orders and valuations. Given a quasi-ordered ring $(R,\preceq),$ we first give a characterization of all Manis valuations (i.e. surjective valuations) $v$ on $R$ that are compatible with $\preceq$ (see Theorem \ref{qoringcomp}). In case where $\preceq$ also comes from a Manis valuation, say $w,$ we will show that $v$ is compatible with $\preceq$ if and only if $v$ is a coarsening of $w$ (see Lemma \ref{coarse}), leading to a characterization of all the Manis coarsenings $v$ of $w$ (see Theorem \ref{coarser}). We conclude this section by developing a notion of rank of a quasi-ordered ring 
(see Definition \ref{qoringrank}).   
\vspace{1mm}

\noindent
In the fourth and final section we establish Baer-Krull Theorems for quasi-ordered rings (see Theorem \ref{BaerKrullI}, respectively Corollary \ref{BaerKrullII}, Theorem \ref{BaerKrullIII}, Theorem \ref{BaerKrullIV}). Once these are proven, we can not only generalize the classical Baer-Krull Theorem to ordered rings (see Corollary \ref{BaerKrullorings}), but also characterize all Manis refinements $w$ of a valued ring $(R,v),$ given that $v$ is also Manis (see Corollary \ref{BaerKrullqoringsII}).

\section{Preliminaries}
\noindent
Here we briefly introduce some basic results concerning valued, ordered and quasi-ordered rings. Moreover, we introduce the theorems, which we aim to establish for quasi-ordered rings (see Theorems \ref{ofieldcomp} and \ref{baerkrullfield}). Throughout this section let $R$ always denote a commutative ring with $1.$

\begin{definition} \label{ringval} (see \cite[VI. 3.1]{bo}) Let $(\Gamma,+,\leq)$ be an ordered abelian group and $\infty$ a symbol such that $\gamma < \infty$ and 
$\infty = \infty + \infty = \gamma + \infty = \infty + \gamma$ for all $\gamma \in \Gamma.$ \\
A map $v: R \to \Gamma \cup \{\infty\}$ is called a \textbf{valuation} on $R$ if for all $x,y \in R:$ \vspace{1mm}

\begin{itemize}
 \item[(V1)] $v(0) = \infty,$ \vspace{1mm}

 \item[(V2)] $v(1) = 0,$ \vspace{1mm}

 \item[(V3)] $v(xy) = v(x) + v(y),$ \vspace{1mm}

 \item[(V4)] $v(x+y) \geq \min\{v(x),v(y)\}.$
\end{itemize} \vspace{1mm}

\noindent
We always assume that $\Gamma$ is the group generated by $\{v(x): x \in v^{-1}(\Gamma)\}$ and call it the \textbf{value group} of $R.$ We also denote it by $\Gamma_v.$ We call $v$ \textbf{trivial} if $\Gamma_v$ is trivial, i.e. if $\Gamma_v = \{0\}.$ The set $\mathfrak{q}_v:= \mathrm{supp}(v) := v^{-1}(\infty)$ is called the \textbf{support} of $v.$
\end{definition}

\begin{fact} \label{fact} \hspace{7cm}
\begin{enumerate}
 \item An easy consequence of the axioms (V1) - (V4) is that $\mathfrak{q}_v$ is a prime ideal of $R.$ \vspace{1mm}

 \item In general, $v$ is not surjective, as $v(R\backslash \mathfrak{q}_v)$ is not necessarily closed under additive inverses. However, if $x$ is a unit, then 
$v(x^{-1}) = -v(x).$ \vspace{1mm}

 \item The subring $R_v := \{x \in R: v(x) \geq 0\}$ of $R$ is said to be the valuation ring of $v.$ The prime ideal $I_v := \{x \in R: v(x) > 0\}$ of $R$ is called the valuation ideal. If $R$ is a field, $R_v$ is a local ring with maximal ideal $I_v.$
\end{enumerate}
\end{fact}

\noindent
We conclude our introduction of valuations with a simple but very helpful lemma.

\begin{lemma} \label{valmin} Let $(R,v)$ be a valued ring and $x,y \in R$ such that $v(x) \neq v(y).$ Then $v(x+y) = \min\{v(x),v(y)\}.$
\begin{proof} Completely analogue as in the field case, see for instance \cite[(1.3.4)]{prestel}. 
\end{proof}
\end{lemma}

\noindent 
Let us now turn towards the notion of orders on rings. For the sake of convenience, they are usually identified with positive cones $P \subset R,$ where $x \in P$ expresses that $x$ is non-negative.

\begin{definition} \label{ringorder} (see \cite[p.29]{marsh}) A \textbf{positive cone} of $R$ is a subset $P \subset R$ such that the following conditions are satisfied:
\begin{itemize}
\item[(P1)] $P \cup -P = R,$ \vspace{1mm}

\item[(P2)] $\mathfrak{p}:= P \cap -P$ is a prime ideal of $R,$ called the \textbf{support} of $R,$ \vspace{1mm}

\item[(P3)] $P \cdot P \subseteq P,$ \vspace{1mm}

\item[(P4)] $P + P \subseteq P.$
\end{itemize}
\end{definition}

\begin{definition} \label{oring} (see \cite[Definition 2.3]{sim}) Let $\leq$ be a binary, reflexive, transitive and total relation on $R.$ Then $(R,\leq)$ is called an \textbf{ordered ring} if for all $x,y,z \in R:$ 
\begin{itemize}
 \item[(O1)] $0 < 1,$ \vspace{1mm}

 \item[(O2)] $xy \leq 0 \Rightarrow x \leq 0 \lor y \leq 0,$ \vspace{1mm}

 \item[(O3)] $x \leq y, \; 0 \leq z \Rightarrow xz \leq yz,$ \vspace{1mm}

 \item[(O4)] $x \leq y \Rightarrow x+z \leq y+z.$
 \end{itemize}
\end{definition}

\noindent
The set of all orders of $R$ is in $1:1$ correspondence with the set of all positives cones of $R$ via $x \leq y \Leftrightarrow y-x \in P.$ Note that if $R$ is a field, then (P2) yields that $\mathfrak{p} = \{0\},$ which precisely means that the corresponding order $\leq$ is anti-symmetric (and vice versa). \vspace{1mm}

\noindent
Recall from the introduction that some quasi-orders $\preceq$ of $R$ are induced by a valuation $v$ via $x \preceq y$ if and only if $v(y) \leq v(x).$ In this case all elements are non-negative. Hence, positive cones are inappropriate to deal with quasi-orders. So in order to compare ordered and quasi-ordered rings, it is necessary to stick to Definition \ref{oring}. \vspace{1mm}

\noindent
Let us now have a closer look at quasi-ordered rings. As mentioned above, Fakhruddin developped a notion of quasi-ordered fields $(K,\preceq)$ and was able to show that quasi-ordered fields are either ordered fields or else $\preceq$ comes from a valuation as above (see \cite[Theorem 2.1]{fakh}). In \cite{sim}, the second author of this paper generalized this result to commutative rings with $1,$ leading to the following result:

\begin{definition} \label{qoring} (see \cite[Definition 3.2]{sim}) Let $R$ be a commutative ring with $1$ and $\preceq$ a binary, reflexive, transitive and total relation on $R.$ If $x,y \in R,$ we write $x \sim y$ if $x \preceq y$ and $y \preceq x,$ and we write $x \prec y$ if $x \preceq y$ but $y \npreceq x.$ \\
The pair $(R,\preceq)$ is called a \textbf{quasi-ordered ring} if for all $x,y,z \in R:$
\begin{itemize}
\item[(QR1)] $0 \prec 1,$ \vspace{1mm}

\item[(QR2)] $xy \preceq 0 \Rightarrow x \preceq 0 \lor y \preceq 0,$ \vspace{1mm}

\item[(QR3)] $x \preceq y, \; 0 \preceq z \Rightarrow xz \preceq yz,$ \vspace{1mm}

\item[(QR4)] $x \preceq y, \; z \nsim y \Rightarrow x+z \preceq y+z,$ \vspace{1mm}

\item[(QR5)] If $0 \prec z,$ then $xz \preceq yz \Rightarrow x \preceq y.$
\end{itemize}
We write $E_x$ for the equivalence class of $x$ w.r.t. $\sim$. $E_0$ is called the \textbf{support} of $\preceq$. 
\end{definition}

\vspace{1mm}

\noindent
In \cite[Theorem 4.6]{sim}, the second author proved that a quasi-ordered ring $(R,\preceq)$ is either an ordered ring or a valued ring $(R,v)$ such that $x \preceq y \Leftrightarrow v(y) \leq v(x).$ Thus, via quasi-ordered rings,
we can treat ordered and valued rings simultaneously.

\begin{remark} \label{rema} \hspace{7cm}
\begin{enumerate}
\item If $(R,\preceq)$ is a quasi-ordered ring with $x \sim 0$ and $y \nsim 0,$ then $x+y \sim y$ (see \cite[Lemma 3.6]{sim}). This result will be useful later on. \vspace{1mm}

\item The support $E_0$ is a prime ideal of $R$ (see \cite[Proposition 3.8]{sim}). \vspace{1mm}

\item The ``new'' axiom (QR5) is crucial for the dichotomy, see \cite[Proposition 3.1]{sim}. Moreover, note that it easily implies (QR2). Indeed, if $xy \preceq 0$ and $0 \prec x,$ then (QR2) yields $y \preceq 0.$ However, we decided to keep axiom (QR2) in order to preserve the analogy between ordered and quasi-ordered rings. \vspace{1mm}

\item If $R$ is a field, then the axioms (QR1) and (QR2) can be replaced with the axiom $x \sim 0 \Rightarrow x = 0,$ while (QR5) becomes unnecessary. As a matter of fact, this is precisely how Fakhruddin introduced quasi-ordered fields in the first place (see \cite[2]{fakh}). 
\end{enumerate}
\end{remark}

\noindent
Later on, we will also use the following variant of axiom (QR5).

\begin{lemma} \label{QR5} Let $(R,\preceq)$ be a quasi-ordered ring and $x,y,z \in R.$ If $z \nsim 0,$ then $xz \sim yz \Rightarrow x \sim y.$
\begin{proof} For $0 \prec z,$ this is the same as (QR5). So suppose that $z \prec 0.$ Then $0 \prec -z.$ Thus, (QR5) tells us $-x \sim -y.$ Assume for a contradiction that $x \nsim y,$ without loss of generality $x \prec y.$ By transitivity of $\sim$ we get either 
$x \nsim -x,-y$ or $y \nsim -x,-y.$ If $y \nsim -x,-y,$ we obtain from $-x \preceq -y$ that $y-x \preceq 0.$ If $x \nsim 0,$ then (QR4) yields $y \preceq x.$ Otherwise, the same follows from Remark \ref{rema}(2). Hence, there is a contradiction anyway. So suppose that 
$x \nsim -x,-y.$ Then $-x \preceq -y$ implies $0 \preceq x-y,$ and $-y \preceq -x$ implies $x-y \preceq 0$ via (QR4). So $x-y \in E_0,$ but then also $y-x \in E_0$ by the previous lemma. Thus, $y-x \sim 0.$ From (QR4) (if $x \nsim 0),$ respectively Remark \ref{rema}(2) 
(if $x \sim 0),$ we obtain $y \preceq x,$ again a contradiction.
\end{proof}
\end{lemma}

\noindent
We conclude this introductive section by recalling the Theorems \ref{compfield} and \ref{baerkrullfield} below, which we will establish for quasi-ordered rings in this paper. So let $(K,\leq)$ be an ordered field. Recall that a valuation $v$ on $K$ is said to be compatible with $\leq,$ if $0 \leq x \leq y$ implies 
$v(y) \leq v(x)$ (see for instance \cite[Definition 2.4]{lam}). A subset $S \subseteq K$ is convex (w.r.t. $\leq$), if from $x \leq y \leq z$ and $x,z \in S$ follows $y \in S.$

\begin{theorem} \label{ofieldcomp} $\mathrm{(see}$ \cite[Theorem 2.3 and Proposition 2.9]{lam} $\mathrm{or}$ \cite[Proposition 2.2.4]{prestel}$\mathrm{)}$  \label{compfield} Let $(K, \leq)$ be an ordered field and let $v$ be a valuation on $K.$ The following are equivalent:
\begin{enumerate}
 \item $v$ is compatible with $\leq,$ \vspace{1mm}

 \item the valuation ring $K_v$ is convex, \vspace{1mm}

 \item the maximal ideal $I_v$ is convex, \vspace{1mm}

 \item $I_v < 1,$ \vspace{1mm}

 \item $\leq$ induces canonically via the residue map $\phi_v: K_v \to Kv:= K_v/I_v, \: x \mapsto x + I_v$ an order $\leq'$ on the residue field $Kv.$ 
\end{enumerate}
\end{theorem}

\noindent
The fifth condition of the previous result is crucial for the second theorem, the so called Baer-Krull Theorem (see \cite[p.37]{prestel}). Let $K$ again be a field and $v$ a valuation on $K$ with value group $\Gamma_v.$ Note that $\overline{\Gamma_v} = \Gamma_v/2\Gamma_v$ is in a canonical way an $\mathbb{F}_2$-vector space. Hence, we find a subset $\{\pi_i: i \in I\} \subset K,$ such that $\{\overline{v(\pi_i)}: i \in I\}$ is an $\mathbb{F}_2$-basis of $\overline{\Gamma_v}.$

\begin{theorem} \label{baerkrullfield} $\mathrm{(}$Baer-Krull Theorem for ordered fields; $\mathrm{see}$ \cite[Theorem 2.2.5]{prestel}$\mathrm{)}$ \\
Let $K$ be a field and $v$ a valuation on $K.$ Moreover, let $\mathcal{X}(K)$ and $\mathcal{X}(Kv)$ denote the set of all orderings on $K,$ respectively $Kv$. There exists a bijective map
\[
\psi: \{\leq \; \in \mathcal{X}(K)\colon \leq \textrm{ is } v\textrm{-compatible}\} \to \{-1,1\}^I \times \mathcal{X}(Kv),
\]
described as follows: given an ordering $\leq$ in the domain of $\psi$, let $\eta_{\leq}: I \to \{-1,1\},$ where $\eta_{\leq}(i) = 1 \Leftrightarrow 0 \leq \pi_i.$ Then the map $\leq \; \mapsto (\eta_{\leq}, \leq')$ is the above bijection, where $\leq'$ denotes the order on $Kv$ from Theorem \ref{ofieldcomp}(5).
\end{theorem}

\section{Compatibility between quasi-orders and valuations}

\noindent
The aim of this section is to prove an analogue of Theorem \ref{compfield} for quasi-ordered rings. First we convince ourselves that for this end, we have to restrict our attention to surjective valuations (see Example \ref{nomanis}), also called Manis valuations. Then we establish that the conditions (1) - (3) and (5) from the said theorem are equivalent for quasi-ordered rings, if $v$ is Manis (see Theorem \ref{qoringcomp}). This gives rise to a characterization of all Manis valuations $w$ on $R,$ which are coarser than $v$ (see Theorem \ref{coarser}). Afterwards, we prove that $I_v \prec 1$ is no equivalent condition anymore, no matter of which of the two kinds the quasi-order is (see Examples \ref{exp1} and \ref{exp2}). Furthermore, we show that Theorem \ref{compfield} holds to the full extend, if we additionally demand that $v$ is local (see Lemma \ref{Iv1}). We conclude this section by establishing the notion of rank of a quasi-ordered ring (see Definition \ref{qoringrank}).

\begin{notation} \label{notation} We use the following notation for the rest of this section:
\begin{enumerate}
 \item Let $R$ always denote a commutative ring with $1.$ If a quasi-order $\preceq$ on $R$ (see Definition \ref{qoring}) is induced by some valuation $w$ on $R,$ we also write $\preceq_w$ instead of $\preceq$ and call it a \textbf{proper quasi-order} (p.q.o). Note that a quasi-order $\preceq$ comes from a valuation if and only if $-1 \prec 0.$ The symbol $\leq$ is reserved for orders. \vspace{1mm}

 \item If $v$ is a valuation on $R,$ we write $\mathfrak{q}_v := \textrm{supp}(v) := v^{-1}(\infty)$ for its \textbf{support} and $\Gamma_v$ for its \textbf{value group} (see Definition \ref{ringval}). Moreover, we denote by $R_v := \{x \in R: v(x) \geq 0\}$ the \textbf{valuation ring} of $v$, by $I_v := \{x \in R: v(x) > 0\}$ the \textbf{valuation ideal}, and by 
$U_v := R_v \backslash I_v := \{x \in R: v(x) = 0\}.$ Last but not least, $Rv := R_v/I_v$ denotes the \textbf{residue class domain} of $v$ and $\phi_v: R_v \to Rv, \; x \mapsto x+I_v$ the \textbf{residue map}.
\end{enumerate}
\end{notation}

\begin{definition} \label{ringcomp} Let $(R,\preceq)$ be a quasi-ordered ring. A valuation $v$ on $R$ is said to be \textbf{compatible} with $\preceq$ (or $\preceq$-compatible), if for all $y,z \in R:$
\[
 0 \preceq y \preceq z \Rightarrow v(z) \leq v(y).
\]
\end{definition}

\noindent
In general, we cannot expect that Theorem \ref{compfield} holds even for ordered rings, as the following basic example shows:

\begin{example} \label{nomanis} \hspace{7cm}
\begin{itemize}
 \item[(1)] Consider the map $v: \mathbb{Z}[X] \to \mathbb{Z} \cup \{\infty\}, \; f = \sum_{i \in \N} a_iX^i \mapsto -\deg f.$ It is easy to verify that $v$ is a valuation on $R.$ We can extend the unique order on $\mathbb{Z}$ to $\mathbb{Z}[X]$ by declaring $0 \leq f :\Leftrightarrow 0 \leq f(0).$ Note that $R_v = \mathbb{Z}$ and $I_v = \{0\},$ so obviously the conditions (4) and (5) of Theorem \ref{compfield} are satisfied. However, the inequalities $0 \leq X \leq 0$ yield that neither 
$I_v$ nor $R_v$ is convex with respect to $\leq.$ Moreover, we have $0 \leq X+1 \leq 1,$ but $v(X+1) = -1 < 0 = v(1),$ so (1) is also not satisfied. \vspace{1mm}

 \item[(2)] Let $p$ be a prime number and $v$ the $p$-adic valuation on the integers $\mathbb{Z},$ i.e. if $x = p^ra_1\ldots a_n$ in the unique prime factorization, then $v(x) = r$ (see \cite[(1.3.1)]{prestel}). Moreover, let $\leq$ denote the unique order on $\mathbb{Z}.$ Then $R_v = \mathbb{Z}$ is convex, so (2) holds. However, it is easy to see that all the other conditions of Theorem \ref{compfield} are not satisfied. 
\end{itemize}
\end{example}

\noindent
As a matter of fact, for a quasi-ordered ring $(R,\preceq)$ and a valuation $v$ on $R,$ in general none of the conditions from Theorem \ref{compfield} is equivalent to another. We do get the following tabular of implications, where \vspace{1mm}

\begin{enumerate}
\item $v$ is compatible with $\preceq,$ \vspace{1mm}

\item $R_v$ is convex, \vspace{1mm}

\item $I_v$ is convex, \vspace{1mm}

\item $I_v \prec 1,$ \vspace{1mm}

\item $\preceq$ induces canonically via the residue map $\phi_v: R_v \to Rv:= R_v/I_v, \: x \mapsto x + I_v$ a quasi-order $\preceq'$ on the residue domain $Rv.$
\end{enumerate}  \vspace{4mm}

\begin{center}
$\begin{array} {c|ccccc}
\Rightarrow & (1) & (2) & (3) & (4) & (5) \\  [1ex]

\hline \\ (1) & & \surd  & \surd & \surd & \surd \\  [1ex]

(2) & & & & & \\  [1ex]

(3) & & & & \surd & \surd \\  [1ex]

(4) & & & & & \\  [1ex]

(5) & & & & \surd &
\end{array}$
\end{center} \vspace{1mm}

\begin{proof} We first show that (1) implies (2) and (3). So let $0 \preceq y \preceq z$ with $z \in R_v.$ Then (1) yields $v(z) \leq v(y).$ Therefore $v(y) \in R_v.$ The same arguments for $I_v$ instead of $R_v.$ In order to show that (1) implies (4) and (5), it suffices to show that (2) implies the two of them. \vspace{1mm}

\noindent
Condition (2) obviously implies (4). It also implies (5) as we will not use the Manis property in the proof of the respective implication in Theorem \ref{qoringcomp}. \vspace{1mm}

\noindent
In Example \ref{nomanis}(ii), we have already seen that (2) implies none of the other conditions. We now show that it does not imply (2), and hece also not (1). Let $v$ be the trivial valuation on $\mathbb{Z}.$ Extend to $\mathbb{Z}[X]$ with $\gamma = -1.$ Let $w$ be the $p$-adic valuation on $\mathbb{Z}.$ Extend to $\mathbb{Z}[X]$ with $\gamma = 1.$ Then (3) holds, because
\[
0 < v(x) \Rightarrow x = 0 \Rightarrow y = 0 \Rightarrow 0 < v(y),
\]
where the first implication follows by definition of $v$ and the second one by (3). \\
However, (2) does not hold. Choose $a = 1$ and $b = X.$ Then $w(a) = 0 \leq w(b) = 1$ and $0 \leq v(a) = 1,$ but $v(b) = -1 < 0.$ \vspace{1mm}

\noindent
Condition (4) implies nothing even if $v$ is Manis, as we will see in Example \ref{exp1} and \ref{exp2}. (5) implies (4): Assume $1 \preceq x$ for some $x \in I_v.$ Then $\overline{1} \preceq \overline{x} = 0,$ but in quasi-ordered rings it holds $0 \prec 1,$ a contradiction. Therefore $I_v \prec 1.$ The fact that (5) does not imply any of the other conditions follows from Example \ref{nomanis}(i).
\end{proof}

\vspace{1mm}

\noindent
Such counterexamples can be prevented by demanding surjectivity of the valuation $v.$ Recall that valuations on fields are automatically surjective. Contrary, in the ring case, $v(R \backslash \mathfrak{q}_v)$ is not necessarily closed under additive inverses, as $R$ is not necessarily closed under multiplicative inverses (see Facts \ref{fact}). Surjectivity will be frequently exploited later on, as it mitigates the lack of multiplicative inverses.

\begin{definition} (see \cite[p.193]{manis}) Let $v$ be a valuation on $R.$ Then $v$ is said to be a \textbf{Manis valuation}, if $v$ is surjective.
\end{definition}

\noindent
We now turn towards the proof of Theorem \ref{compfield} for quasi-ordered rings. This requires some preliminaries.

\begin{definition} Let $(R,\preceq)$ be a quasi-ordered ring and $S \subseteq R$ a subset of $R.$ Then $S$ is said to be \textbf{convex}, if $x \preceq y \preceq z$ and $x,z \in S$ implies $y \in S.$
\end{definition}

\noindent
The following lemma simplifies convexity in a usual manner and holds particularly for the valuation ring $R_v$ and its prime ideal $I_v,$ as $v(x) = v(-x)$ for all $x \in R.$

\begin{lemma} \label{...} Let $(R,\preceq)$ be a quasi-ordered ring. A subset $S \subseteq R$ with $0 \in S$ and $S = -S$ is convex, if and only if $0 \preceq y \preceq z$ and $z \in S$ implies $y \in S$.
\begin{proof} The implication $\Rightarrow$ is trivial. So suppose that the right hand side holds and let $x \preceq y \preceq z$ with $x,z \in S.$ If $0 \preceq y,$ it follows immediately by assumption that $y \in S.$ So suppose that $y \prec 0.$ Then $x \preceq y \prec 0.$ We will show $0 \prec -y \preceq -x.$ Note that 
$-x \in S$ because $S = -S.$ Hence, we obtain $-y \in S,$ but then also $y \in S.$ \vspace{1mm}

\noindent
Clearly $0 \prec -x,-y$ by axiom (QR4) and the fact that $E_0$ is an ideal (see Remark \ref{rema}(2)). It remains to show that $-y \preceq -x.$ Assume for a contradiction $-x \prec -y.$ Note that $y \prec 0 \prec -x,-y,$ therefore $-x \not\sim y$ and $y \not\sim -y.$ Via (QR4), it follows from $x \preceq y$ that $0 \preceq y-x$ and from $-x \preceq -y$ that $y-x \preceq  0.$ Thus, $y-x \in E_0.$ This implies $-y \sim -x$ (see Remark \ref{rema}(2)), a contradiction.
\end{proof}
\end{lemma}

\noindent
The most difficult part of the proof will be to show that if $v$ is a $\preceq$-compatible valuation on $R,$  then $\preceq$ induces a quasi-order on the residue class domain $Rv.$ For that purpose we want to exploit convexity of $I_v.$

\begin{lemma} \label{convexity} Let $(R,\preceq)$ be a quasi-ordered ring, $v$ a valuation on $R$ such that $I_v$ is convex, and $u \in U_v.$ 
\begin{enumerate}
 \item If $c \in I_v,$ then $c \not\sim u.$ \vspace{1mm}

 \item If $0 \prec u,$ then $0 \prec u +c$ for all $c \in I_v.$ \vspace{1mm}

 \item If $u \prec 0,$ then $u+c \prec 0$ for all $c \in I_v.$
\end{enumerate}
\begin{proof} \hspace{7cm} 
\begin{enumerate}
 \item Assume $c \sim u.$ Then $c \preceq u \preceq c$ and convexity of $I_v$ yields $u \in I_v,$ a contradiction. \vspace{1mm}

 \item Assume that $0 \prec u$ but $0 \nprec u+c$ for some $c \in I_v.$ Then $0 \prec u$ and $u+c \preceq 0.$ Note that this implies $c \notin E_0,$ as otherwise $u \sim u+c$ (see Remark \ref{rema}(1)). Hence, we obtain $u \preceq -c.$ So it holds $0 \prec u \preceq -c.$ Convexity of $I_v$ yields $u \in I_v,$ a contradiction. 
\vspace{1mm}

 \item Assume $0 \preceq u+c$ for some $c \in I_v.$ Then $u \prec 0 \preceq u+c.$ It remains to show that $-u \not\sim u+c.$ Then $0 \prec -u \preceq c$ and one may conclude by convexity of $I_v.$ So assume for a contradiction that $-u \preceq u+c.$ From Lemma \ref{valmin} follows $u+c \in U_v,$ so (1) yields that $-c \not\sim u+c.$ Thus, one obtains $-u-c \preceq u.$ Now note that $0 \prec -u \in U_v.$ So (2) yields $0 \prec -u-c.$ Therefore $0 \prec -u-c \preceq u \prec 0,$ a contradiction. This finishes the proof.
\end{enumerate}
\end{proof}
\end{lemma}

\noindent
Moreover, we require a couple of results that Fakhruddin established in the more specific setting of quasi-ordered fields (see \cite{fakh}).

\begin{lemma} \label{lemmaxxx} Let $(R,\preceq)$ be a quasi-ordered ring and $x \in R.$ Then $x \sim -x$ if and only if $0 \preceq x,-x.$
\begin{proof} Just as in the case of quasi-ordered fields, see \cite[Lemma 3.1]{fakh}.
\end{proof}
\end{lemma}

\begin{lemma} Let $(R,\preceq)$ be a quasi-ordered ring and $x,y \in R.$ If $x \sim y,$ then $x \sim -y$ or $0 \sim x-y$ .
\begin{proof}  If $x,y \sim 0,$ then $x \sim -y,$ as $E_0$ is an ideal. So suppose that $x,y \nsim 0,$ and assume that $x \sim -y.$ We show $0 \sim x-y.$ Note that $y \preceq x \nsim -y.$ Therefore $0 \preceq x-y.$ Moreover, $x \preceq y \sim x \nsim -y,$ so $y \nsim -y,$ and therewith $x-y \preceq 0.$ Thus, $0 \sim x-y.$ 
\end{proof}
\end{lemma}

\begin{corollary} Let $(R,\preceq)$ be a quasi-ordered ring. Then $\sim$ is preserved under multiplication, i.e. if $x,y,a \in R$ such that $x \sim y,$ then $ax \sim ay.$ \label{corpres}
\begin{proof} The cases $0 \preceq a$ (axiom (QR3)) and $x,y$ in $E_0$ ($E_0$ is an ideal) are both trivial. So suppose that $0 \npreceq a$ and $x,y \nsim 0.$ Then $0 \preceq -a.$ The previous lemma gives rise to a case distinction. First suppose $0 \sim x-y.$ Since $-x\nsim 0$ it holds $-x \nsim x-y.$ Hence, $0 \sim x-y$ yields $-x \sim -y.$ Since $0 \preceq -a,$ axiom (QR3) yields $ax \sim ay.$ Now suppose that $0 \nsim x-y.$ Then also $0 \nsim y-x.$ The previous lemma implies $x \sim -y$ and $y \sim -x.$ Therefore $-y \sim x \sim y \sim -x.$ Since $0 \preceq -a,$ we obtain $ay=(-a)(-y) \sim (-a)(-x) = ax.$
\end{proof}
\end{corollary}

\begin{lemma} Let $(R,\preceq)$ be a quasi-ordered ring such that $0 \prec -1.$ Then it holds $x+y \preceq \max\{x,y\}$ for all $x,y \in R.$ \label{lemmaineq}
\end{lemma} 
\begin{proof}
Basically as in the field case, see \cite[Lemma 4.1]{fakh}. Suppose that $x \preceq y,$ and assume for a contradiction that $y \prec x+y.$ Note that $0 \preceq 1$ by axiom (QR1). Lemma \ref{lemmaxxx} implies $-1 \sim 1,$ so the previous corollary yields $-r \sim r$ for all $r \in R.$ It follows $-x \sim x \preceq y \prec x+y.$ Particularly, $-y \nsim x+y,$ since $y \nsim x+y.$ So, by applying (QR4), we obtain $x+y \sim -x-y \preceq x \preceq y,$ a contradiction. 
\end{proof}

\noindent
Finally, we can prove the main theorem of this section:

\begin{theorem} \label{qoringcomp} Let $(R,\preceq)$ be a quasi-ordered ring and let $v$ be a Manis valuation on $R.$ \vspace{1mm}

\begin{itemize}
\item[(a)] The following are equivalent: \vspace{1mm}

\begin{itemize}
 \item[(1)] $v$ is compatible with $\preceq.$ \vspace{1mm}

 \item[(2)] $I_v$ is convex. \vspace{1mm}

 \item[(3)] $\preceq$ induces canonically via the residue map $x \mapsto x + I_v$ a quasi-order $\overline{\preceq}$ with support $\{0\}$ on the residue class domain $Rv.$
\vspace{1mm}

\end{itemize}

\noindent
Moreover, any of these conditions implies that $R_v$ is convex. \vspace{1mm}

\item[(b)] If $v$ is non-trivial, then \vspace{1mm}

\begin{itemize}
\item[(4)] $R_v$ is convex \vspace{1mm}

\end{itemize}
is equivalent to the conditions $\mathrm{(1) - (3)}$.
\end{itemize}
\end{theorem}
\begin{proof} \hspace{7cm}
\begin{itemize}
\item[(a)] We first prove that (1) and (2) are equivalent. So suppose that (1) holds and let $0 \preceq y \preceq z$ with $z \in I_v.$ Then (1) yields that $0< v(z) \leq v(y),$ and therefore $y \in I_v.$ Now suppose that $I_v$ is convex and assume for a contradiction that there exist some $0 \preceq y \preceq z$ such that 
$v(y) < v(z).$ Note that $y \notin \mathfrak{q}_v.$ Hence, since $v$ is Manis, we find some $0 \preceq a$ such that $v(a) = -v(y)$ (for if $a \prec 0,$ then $0 \preceq -a$ and $v(a) = v(-a)$). Via axiom (QR3) follows $0 \preceq ay \preceq az$ with $v(ay) = 0$ and $v(az) = v(z)-v(y) > 0,$ so $az \in I_v$ but $ay \notin I_v.$ This contradicts the convexity of $I_v.$ 
\vspace{1mm}

\noindent
We continue by showing that (2) and (3) are equivalent. First suppose that (3) holds and let $0 \preceq y \preceq z$ with $z \in I_v.$ Assume for a contradiction $y \notin I_v.$ Choose $a \in R$ with $0 \preceq a$ and $v(a) = -v(y).$ Then $0 \preceq ay \preceq az$ with $v(ay) = 0$ and $v(az) > 0.$ Taking residues, it follows $0 \preceq' \overline{ay} \preceq' 0.$ Since the support of $\preceq'$ is trivial, this yields that $\overline{ay} = 0,$ contradicting $v(ay) = 0,$ i.e. $ay \notin I_v.$ Therefore, $y \in I_v.$ \vspace{1mm}

\noindent
Now suppose that (2) holds. The quasi-order induced by the residue map is given by
$$\overline{x} \preceq' \overline{y} :\Leftrightarrow \exists c_1,c_2 \in I_v: x+c_1 \preceq y+c_2.$$
First of all we verify that $\preceq'$ is well-defined. So assume that $\overline{x} \preceq' \overline{y},$ and let $\overline{x} = \overline{x_1}$ and $\overline{y} = \overline{y_1},$ say $x = x_1 + c_1$ and $y = y_1 +c_2$ for some $c_1,c_2 \in I_v.$ There exist some $c_3,c_4 \in I_v$ such that $x+c_3 \preceq y+c_4.$ But then $x_1+(c_1+c_3) \preceq y_1+(c_2+c_4),$ thus, $\overline{x_1} \preceq' \overline{y_1}.$ \vspace{1mm}

\noindent
Evidently, $\preceq'$ is reflexive and total. Next we show transitivity. So assume that $x+c_1 \preceq y+c_2$ and $y+d_1\preceq z+d_2$ for some $c_1,c_2,d_1,d_2 \in I_v.$ We argue by case distinction. First suppose that $y \in U_v.$ Assume for a contradiction that $x+e_1 \succ z+e_2$ for all $e_1,e_2 \in I_v.$ In particular, 
$x+c_1+d_1-c_2 \succ z+d_2.$ Note that $y+c_2 \in U_v$ and $d_1-c_2 \in I_v,$ so Lemma \ref{convexity}(1) yields $y+c_2 \nsim d_1-c_2.$ So from the inequality 
$x+c_1 \preceq y+c_2$ follows $$x+c_1+d_1-c_2 \preceq y+c_2+d_1-c_2 = y+d_1 \preceq z+d_2,$$ a contradiction. \vspace{1mm}

\noindent
If $y \in I_v,$ then $y+c_2$ and $y+d_1 \in I_v.$ By convexity and Lemma \ref{convexity}(2) and (3), this yields that $x$ is either a negative unit or in $I_v,$ and that $z$ is either a positive unit or in $I_v.$ We only have to consider the case where both elements are in the valuation ideal. But then $x+(z-x) \preceq z+0,$ and thus $\overline{x} \preceq' \overline{z}.$ \vspace{1mm}

\noindent
Now we show that the support of $\preceq'$ equals $\{0\}.$ So let $\overline{x} \sim 0$ and assume for a contradiction that $x \in R_v \backslash I_v = U_v.$ Then there exist $c_1,c_2 \in I_v$ such that $x+c_1 \preceq c_2$ and there exist $d_1,d_2 \in I_v$ such that $d_1 \preceq x+d_2.$ \\ 
If $0 \prec x,$ then $0 \prec x + c_1 \preceq c_2$ by Lemma \ref{convexity}(2), and we have $x+c_1 \in I_v$ by convexity, a contradiction. Likewise, if $x \prec 0,$ then $d_1 \preceq x+d_2 \prec 0$ by Lemma \ref{convexity}(3), again contradicting the convexity. \vspace{1mm}

\noindent
It remains to check the axioms (QR1) and (QR3) - (QR5) (axiom (QR2) is omitted because of Remark \ref{rema}(3)). 
\begin{itemize} \vspace{1mm}

 \item[(QR1)] Assume for a contradiction that $\overline{1} \preceq' \overline{0}.$ Then there exist $c_1,c_2 \in I_v$ such that $0 \prec 1+c_1 \preceq c_2$ 
(Lemma \ref{convexity}(2)). Convexity of $I_v$ yields $1+c_1 \in I_v,$ and therefore $1 \in I_v,$ a contradiction. Thus, $\overline{0} \prec' \overline{1}.$ 
\vspace{1mm}

 \item[(QR3)] We have to verify that $0 \preceq' \overline{x}$ and $\overline{y} \preceq' \overline{z}$ implies $\overline{xy} \preceq' \overline{xz}.$ For $\overline{x} = 0,$ there is nothing to show, so assume without loss of generality $x \notin I_v.$ From $0 \preceq' \overline{x}$ follows that there are some $c_1,c_2 \in I_v$ such that $c_1 \preceq x+c_2.$ Applying Lemma \ref{convexity}(1) yields that $c_1-c_2 \preceq x.$ So convexity of $I_v$ gives us $0 \preceq x.$ Moreover, 
$\overline{y} \preceq' \overline{z}$ means $y+d_1 \preceq z+d_2$ for some $d_1,d_2 \in I_v.$ (QR3) implies $xy+xd_1 \preceq xz+xd_2,$ and therefore 
$\overline{xy} \preceq \overline{xz}.$ \vspace{1mm}

\item[(QR4)] We have to prove that $\overline{x} \preceq' \overline{y}$ and $\overline{y} \not\sim \overline{z}$ yields $\overline{x+z} \preceq' \overline{y+z}.$ Let $c_1,c_2 \in I_v$ such that $x+c_1 \preceq y+c_2.$ Note that $\overline{y} \not\sim \overline{z}$ implies either 
$\forall e_1,e_2: y+e_1 \prec z+e_2$ or $\forall e_1,e_2: y+e_1 \succ z+e_2.$ Either way, $z \not\sim y+c_2.$ But then $x+z+c_1 \preceq y+z+c_2$ by (QR4), i.e. 
$\overline{x+z} \preceq' \overline{y+z}.$ \vspace{1mm}

\item[(QR5)] We have to show that if $0 \prec' \overline{a},$ then $\overline{ax} \preceq' \overline{ay}$ implies $\overline{x} \preceq' \overline{y}.$ Note that if 
$ax \preceq ay,$ then $x \preceq y$ by axiom (QR5), hence $\overline{x} \preceq' \overline{y}.$ So from now on assume that $ay \prec ax.$ First we show that one may also assume that $x,y \in U_v.$ Indeed, suppose that $\overline{x} = 0$ and assume for a contradiction that $\overline{y} \prec' 0.$ Then $\overline{ay} \preceq' 0$ by axiom (QR3). But equality cannot hold because neither $a \in I_v,$ nor $y \in I_v.$ Thus, $\overline{ay} \prec' 0 = \overline{ax},$ contradicting the assumption. Now suppose that $\overline{y} = 0$ and assume for a contradiction that $0 \prec' \overline{x}.$ Then $\overline{ay} = 0 \prec \overline{ax},$ again a contradiction. Hence, one may assume that both $x$ and $y$ lie in $U_v.$ So from $\overline{ax} \preceq' \overline{ay}$ follows that there exists some $c \in I_v$ such that 
$ax \preceq ay+c.$ Thus, it holds $ay \prec ax \preceq ay+c.$ The rest of the proof is done by case distinction. \\
If $0 \prec -1,$ then $0 \preceq -r$ for all $r \in R$ with $0 \preceq r$ by (QR3). This yields that all elements are non-negative. Particularly, since $ay$ is a unit and $I_v$ is convex, it holds $c \prec ay$ (otherwise $ 0 \prec ay \preceq c \in I_v$). From Lemma \ref{lemmaineq} follows 
$ay \prec ay+c \preceq \max\{ay+c\} = ay,$ the desired contradiction. \vspace{1mm}

\noindent
Finally suppose $-1 \prec 0.$ Consider the inequalities $ay \prec ax \preceq ay+c.$ By Lemma \ref{convexity}(2) and (3), $ay$ and $ay+c$ have the same sign, and so 
$ax$ has also the same sign, which is contrary to the sign of $-ay.$ Particularly, we may add $-ay$ to these two inequalities and obtain $0 \preceq a(x-y) \preceq c.$ By convexity of $I_v$ follows $a(y-x) \in I_v$ and since $I_v$ is a prime ideal with $a \notin I_v,$ one obtains $\overline{x} = \overline{y}.$ Particularly, 
$\overline{x} \preceq' \overline{y},$ as desired. 
\end{itemize} \vspace{1mm}

\noindent
The convexity of $R_v$ follows immediately from (1), just like the convexity of $I_v.$ \vspace{1mm}

\item[(b)]
It suffices to show that (4) implies (2). So let $0 \preceq y \preceq z$ with $z \in I_v.$ Assume for that $y \notin I_v,$ so by convexity of $R_v$ it holds 
$y \in R_v-I_v = U_v,$ i.e. $v(y) = 0.$ Since $z \in I_v,$ we get $\gamma := v(z) > 0.$ We distinguish the two cases, whether $z \in \mathfrak{q}_v$ or not. \\
If $z \notin \mathfrak{q}_v,$ there exists some $0 \preceq a \in R$ such that $v(a) = -\gamma < 0.$ Axiom (QR3) yields $0 \preceq ay \preceq az.$ As $0$ and $az$ lie in $R_v,$ it follows by convexity of $R_v$ that $ay \in R_v,$ i.e. $v(ay) \geq 0.$ However, $v(ay) = v(a) + v(y) < 0,$ a contradiction. \\
If $z \in \mathfrak{q}_v,$ choose some $0 \preceq a \in R$ with $v(a) < 0,$ which exists since $v$ is a non-trivial Manis valuation. Then $0 \preceq ay \preceq az$ with $az \in R_v.$ However, $ay \notin R_v,$ contradicting the convexity of $R_v.$
\end{itemize}
\end{proof}

\begin{remark} \hspace{7cm}

\begin{enumerate} \label{qoringcompapp}
 \item The assumption in (b) that $v$ is non-trivial is crucial, no matter which kind of a quasi-order $\preceq$ is. \\
For the ordered case consider $\mathbb{Z}$ with its unique order and the trivial valuation $v$ mapping the even integers to $\infty$ and the odd integers to $0.$ Then $R_v = \mathbb{Z}$ is convex, while $I_v = 2\mathbb{Z}$ is not. \vspace{1mm}

\noindent
In the case $\preceq = \preceq_w,$ take the same $v$ and let $w$ be the $p$-adic valuation on $\mathbb{Z}$ for some prime $p>2.$ Then $R_v$ is clearly convex. However, choosing $y = 2$ and $z = 1$ yields $0 = w(y) \leq w(z) = 0$ and $0 < v(y) = \infty,$ but $0 = v(z).$ \vspace{1mm}

 \item Instead of $v$ non-trivial, one may have also demanded that $\mathfrak{q}_v = E_0$ for part (b), i.e. that the supports coincide. Then $z \in \mathfrak{q}_v$ yields $z \in E_0,$ so also $y \in E_0 = \mathfrak{q}_v \subseteq I_v$ by transitivity of $\preceq.$  \vspace{1mm}

 \item $I_v \prec 1$ (compare Theorem \ref{compfield}) is an easy consequence of these conditions. It follows for instance immediately from the convexity of $I_v.$ 
\vspace{1mm}

 \item If $\preceq$ is an order (respectively a proper quasi-order), then $\preceq'$ is also an order (respectively a proper quasi-order).
\begin{proof} First suppose that $\preceq$ is an order. Comparing the definitions of ordered rings (Definition \ref{oring}) and quasi-ordered rings (Definition \ref{qoring}), we only have to show that $\preceq'$ is compatible with $+.$ From $\overline{x} \preceq' \overline{y}$ follows $x+c_1 \preceq y+c_2$ for some 
$c_1,c_2 \in I_v.$ Since $\preceq$ is an order, we get $x+z+c_1 \preceq y+z+c_2,$ thus, $\overline{x+z} \preceq' \overline{y+z}.$ So $\preceq'$ is indeed an order. \vspace{1mm}

\noindent
Finally, suppose that $\preceq = \preceq_w$ for some valuation $w$ on $R.$ We consider the map $w/v: Rv \to \Gamma_{w/v} \cup \{\infty\}$ given by
\[
 w/v(a+I_v) := \begin{cases} \infty &\mbox{if } a \in I_v \\
w(a) & \mbox{else} \end{cases}.
\] 
(compare \cite[p.45]{prestel} for the field case). We prove that $w/v$ is well-defined. For $a \in I_v$ this is clear by definition. So suppose that $a \in U_v$ and $c \in I_v.$ We have to show that $w(a) = w(a+c).$ From condition (1) of the previous theorem we obtain for all $x,y \in R$ that if $w(x) \leq w(y),$ then 
$v(x) \leq v(y).$ Hence, it follows from $v(a) = 0 < v(c)$ that also $w(a) < w(c).$ Lemma \ref{valmin} yields $w(a+c) = \min\{w(a),w(c)\} = w(a).$ \vspace{1mm}

\noindent
It is easy to see that $w/v$ satisfies the axioms (V1) and (V2) from Definition \ref{ringval}. For (V3) note that $ab \in I_v$ if and only if $a \in I_v$ or $b \in I_v,$ since $I_v$ is prime, so $w/v(ab+I_v) = \infty$ if and only if $w/v(a+I_v) + w/v(b+I_v) = \infty.$ From this observation (V3) is easily deduced. The proof of (V4) is done by a similar case distinction. Hence, $w/v$ defines a valuation on $Rv.$ Its support is $\{0\},$ as $\mathfrak{q}_w \subseteq \mathfrak{q}_v \subseteq I_v,$ which again follows from Theorem \ref{qoringcomp}(1). Moreover, for $x,y \in U_v$ (i.e. $\overline{x},\overline{y} \neq 0$) it holds
\begin{align*}
\overline{x} \preceq_w' \overline{y} &\Leftrightarrow x+c_1 \preceq_w y+c_2 \textrm{ for some } c_1,c_2 \in I_v \\ & \Leftrightarrow w(y+c_2) \leq w(x+c_1) \textrm{ for some } c_1,c_2 \in I_v 
\\ & \Leftrightarrow w(y) \leq w(x) \\  &\Leftrightarrow w/v(\overline{y}) \leq w/v(\overline{x}),
\end{align*}
where the third equivalence follows precisely as in the proof of the well-definedness of $w/v,$ while the last equivalence is just the definition of $w/v.$ This proves that $\preceq_w' = \preceq_{w/v}.$
\end{proof}
\end{enumerate}
\end{remark}

\noindent
If $\preceq$ is an order, then Theorem \ref{qoringcomp} generalizes Theorem \ref{compfield} from ordered fields to ordered rings. Next we show that if $\preceq = \preceq_w$ for some Manis valuation $w,$ then Theorem \ref{qoringcomp} precisely characterizes the Manis valuations $v$ on $R$ that are coarser than $w.$

\begin{definition} \label{coarsening} (see \cite[p.415]{grif}) Let $v,w$ be valuations on $R.$ Then $v$ is said to be a \textbf{coarsening} of $w$ (or $w$ a \textbf{refinement} of $v$), in short, $v \leq w,$ if there exists an order homomorphism $\phi: \Gamma_w \to \Gamma_v$ such that $v = \phi \circ w,$ or equivalently, if 
$R_w \subseteq R_v$ and $I_v \subseteq I_w.$ 
\end{definition}

\begin{lemma} \label{powerssup} Let $v \leq w$ be non-trivial Manis valuations on $R.$ Then $q_v = q_w.$
\begin{proof} This is part of \cite[Proposition 3.1]{powers}.
\end{proof}
\end{lemma}

\noindent
Actually, Power's proof of the previous result only requires that $v$ is non-trivial. However, from $v$ non-trivial and $v \leq w$ follows immediately that $w$ is also non-trivial.

\begin{lemma} \label{coarse} Let $v$ and $w$ be non-trivial Manis valuations on $R.$ The following are equivalent:
\begin{enumerate}
 \item $v$ is $\preceq_w$-compatible $\mathrm{(}$i.e. $w(y) \leq w(z) \Rightarrow v(y) \leq v(z)\mathrm{)}.$ \vspace{1mm}

 \item $v$ is a coarsening of $w.$
\end{enumerate}
\begin{proof} We first show that (1) implies (2). Let $x \in R_w.$ Then $0 = w(1) \leq w(x),$ so also $0 = v(1) \leq v(x),$ thus $x \in R_v.$ Likewise, if $x \notin I_w,$ then $w(x) \leq w(1) = 0,$ which yields that $v(x) \leq v(1) = 0.$ Therefore $x \notin I_v.$ \vspace{1mm}

\noindent
Conversely, assume that (2) holds and suppose that $w(y) \leq w(z).$ By the previous lemma we get $\mathfrak{q}_w = \mathfrak{q}_v,$ so we may assume that $y$ is not in the support of these valuations. Moreover note that 
$U_w \subseteq U_v;$ indeed, if $u \in U_w,$ then $u \in R_w$ and $u \notin I_w,$ thus $u \in R_v$ and $u \notin I_v.$ Therefore, $u \in U_v.$ As $w(y) \in \Gamma_w$ and $w$ is Manis, there exists some $a \in R$ such that 
$w(a) = -w(y).$ It follows $ay \in U_w$ and $az \in R_w.$ Therefore, $ay \in U_v$ and $az \in R_v.$ It is easy to see that this implies $v(y) \leq v(z).$
\end{proof}
\end{lemma}

\noindent
Hence, we obtain as a special case of Theorem \ref{qoringcomp} the following characterization of coarsenings of $v$:

\begin{theorem} \label{coarser} Let $v,w$ be non-trivial Manis valuations on $R.$ Then $v$ is a coarsening of $w,$ if and only if one of the following equivalent conditions is satisfied for all $x,y \in R:$
\begin{enumerate}
 \item $w(x) \leq w(y) \Rightarrow v(x) \leq v(y),$ \vspace{1mm}

 \item $w(x) \leq w(y), 0 \leq v(x) \Rightarrow 0 \leq v(y),$ \vspace{1mm}

 \item $w(x) \leq w(y), 0 < v(x) \Rightarrow 0 < v(y),$ \vspace{1mm}

 \item $w/v: Rv \to \Gamma_{w/v} \cup \{\infty\}, \; x+I_v \mapsto \begin{cases} \infty &\mbox{if } x \in I_v \\
w(x) & \mbox{else} \end{cases}$ \quad defines a valuation with support $\{0\}.$
\end{enumerate}
\end{theorem}
\begin{proof} This is precisely Theorem \ref{qoringcomp} in the case where the quasi-order $\preceq$ comes from a Manis valuation $w,$ and Lemma \ref{coarse}. Moreover, we simplified the convexity of $R_v$ and $I_v$ (in (2) and (3)) according to Lemma \ref{...}. 
\end{proof}

\noindent
Next we show that $I_v \prec 1$ is not equivalent to all the other conditions of Theorem \ref{qoringcomp}, regardless of whether $\preceq$ is a proper quasi-order (Example \ref{exp1}) or an order (Example \ref{exp2}), even if $v$ is non-trivial.

\begin{theorem} Let $R$ be a ring, $\Gamma \subseteq \Gamma'$ ordered abelian groups, $u: R \to \Gamma \cup \{\infty\}$ a valuation on $R,$ and $\gamma \in \Gamma'.$ For $f = \sum_{i=0}^{n} a_i X^i \in R[X]$ define
\[
v(f) = \begin{cases} \infty &\mbox{if } f = 0 \\
\min\limits_{0\leq i \leq n}\{u(a_i) + i\gamma\} & \mbox{otherwise }. \end{cases}
\]
Then $v: R[X] \to \Gamma' \cup \{\infty\}$ is a valuation that extends $u.$
\end{theorem}
\begin{proof} As in the field case, see \cite[Theorem 2.2.1]{prestel}.
\end{proof}

\begin{example} \label{exp1} Let $v_p: \mathbb{Q} \to \mathbb{Z} \cup \{\infty\}$ denote the $p$-adic valuation for some prime number $p \in \N$ (see \cite[p.18]{prestel}), i.e. if 
$0 \neq x = \frac{p^ra_i}{p^sb_i} \in \Q$ (using the unique prime factorization in $\mathbb{Z}$), then $v_p(x) = r-s.$ Apply the previous theorem with $\gamma = 1$ to extend $v_p$ to a valuation $v: \Q[X] \to \Z \cup \{\infty\}.$ The valuation $v$ is Manis, as $v_p$ is Manis and they have the same value group. We do the same procedure with $w$ instead of $v,$ except that this time $\gamma = 0.$ \vspace{1mm}

\noindent
Note that $v = w$ on $\mathbb{Q}$ and $v(f) = w(f)+ i$ for some $i \geq 0$ if $f \in \mathbb{Q}[X] \backslash \mathbb{Q}.$  This implies $I_v \prec_w 1.$ Indeed, $f \in I_v$ means $v(f) > 0.$ But then also $w(f) > 0  = w(1),$ and therefore $f \prec_w 1.$ However, $v$ is not compatible with $\preceq_w.$ For instance we have 
$w(X^2) = 0 < w(p) = 1 < w(0) = \infty,$ but $v(p) = 1 < v(X^2) = 2.$
\end{example}

\begin{example} \label{exp2} Consider the trivial valuation $u(x) = 0$ for $x \neq 0$ on $\mathbb{Z}.$ Extend $u$ via the previous theorem to a valuation $v$ on $\mathbb{Z}[X,Y]$ with $\gamma = 1$ (for $X$), respectively $\gamma = -1$ (for $Y$). Thus, for any $0 \neq f = \sum\limits_{i,j} a_{ij}X^iY^j \in \Z[X,Y],$ we have 
$$v(f) = \min_{i,j}\{v(a_{ij})+i-j\}.$$ Note that $v$ is a Manis valuation with value group $\mathbb{Z},$ for if $m$ is an integer, then either $v(X^m) = m$ 
(if $m \geq 0$) or $v(Y^{-m}) = m$ (else). Order $\mathbb{Z}[X,Y]$ by declaring $f \geq 0 :\Leftrightarrow f(0) \geq 0.$ Note that $v(f) \leq 0$ if $a_{00} \neq 0.$ Therefore, $I_v \subseteq \langle X,Y \rangle \subseteq E_0,$ so $I_v < 1.$ However, $I_v$ is not convex since $0 \leq Y \leq 0,$ but $Y \notin I_v.$
\end{example}

 
\begin{remark} In the case of ordered fields $(K,\leq),$ the condition $I_v < 1$ is often times replaced with the equivalent condition $1+I_v \geq 0$ (see for instance \cite[Definition 2.4]{lam} or \cite[Proposition 2.2.4]{prestel}). Note, however, that this is inappropriate for proper quasi-orders, as $1+I_v \succeq_w 0$ is then trivially satisfied.  
\end{remark}

\noindent
We continue this section by imposing a suitable extra condition on $v,$ such that $I_v \prec 1$ is equivalent to (1) - (3) from Theorem \ref{qoringcomp}.

\begin{definition} (see \cite[Ch. I, Definition 5]{kneb}) A valuation $v$ on $R$ is called \textbf{local}, if the pair $(R_v,I_v)$ is local, i.e. if $I_v$ is the unique maximal ideal of $R_v.$
\end{definition}

\noindent
The maximal ideal of a local ring consists precisely of all non-units of the said ring. A characterization of local valuations is given in \cite[Ch. I, Proposition 1.3]{kneb} and \cite[Proposition 5]{grif}, respectively. If $R$ is a field, then $v$ is always a local Manis valuation.

\begin{lemma} \label{Iv1} Let $(R,\preceq)$ be a quasi-ordered ring and $v$ a local Manis valuation on $R.$ The following are equivalent:
\begin{enumerate}
 \item $v$ is compatible with $\preceq.$ \vspace{1mm}

 \item $I_v \prec 1.$
\end{enumerate}
\begin{proof} 
(1) implies (2) is clear, see Remark \ref{qoringcompapp}(3). Now suppose that (2) holds, and assume for a contradiction that there are some $y,z \in R$ such that $0 \preceq y \preceq z,$ but $v(y) < v(z).$ The latter implies $y \notin \mathfrak{q}_v.$ Since $v$ is Manis, we find some $0 \preceq a$ such that $v(a) = -v(y).$ We obtain $0 = v(ay) < v(az),$ so $ay \in U_v$ and $az \in I_v.$ As $v$ is local and $ay \in U_v,$ $ay$ is a unit. It follows 
$$0 < v(az) - v(ay) = v\left(\frac{az}{ay}\right),$$ i.e. $\frac{az}{ay} \in I_v.$ Hence, (2) yields $\frac{az}{ay} \prec 1.$ This implies $az \prec ay$ 
(for if $az \sim ay,$ then $\frac{az}{ay} \sim 1$ by Corollary \ref{corpres}, a contradiction). On the other hand, it follows from $y \preceq z$ and $0 \preceq a$ that $ay \preceq az,$ a contradiction. 
\end{proof}
\end{lemma}

\begin{corollary} Let $v,w$ be non-trivial Manis valuations on $R$ such that $v$ is local. Then $v$ is coarser than $w$ if and only if $I_v \subseteq I_w.$
\begin{proof} This is an immediate consequence of Lemma \ref{coarse} and Lemma \ref{Iv1} in the case where $\preceq = \preceq_w$ for some non-trivial Manis valuation $w.$
\end{proof}

\end{corollary}

\noindent
We conclude this section by establishing a notion of rank of a quasi-ordered ring. For the sake of convenience we first consider the field case and then reduce the ring case to it.

\begin{definition} (compare \cite[Ch. I, Definition 2]{kneb}) Two valuations $v,w$ on $R$ are said to be \textbf{equivalent}, in short, $v \sim w,$ if $v(x) \leq v(y) \Leftrightarrow w(x) \leq w(y)$ for all $x,y \in R.$
\end{definition}

\noindent
Let $v,w$ be two non-trivial Manis valuations on $R.$ From Theorem \ref{coarser}(1) and Definition \ref{coarsening} follows
$$v \sim w \Leftrightarrow v \leq w \textrm{ and } w \leq v \Leftrightarrow R_v = R_w \textrm{ and } I_v = I_w.$$ By abuse of language, we identify equivalent valuations (this is quite common in the literature, see for instance \cite[p.11]{kneb} or \cite[p.256]{powers}).

\begin{proposition} \label{lemmarank} Let $(K,\preceq)$ be a quasi-ordered field. The set
$$\mathcal{R}:= \{w: w \textrm{ is a non-trivial } \preceq\textrm{-compatible valuation on } K\}$$ is totally ordered by $\leq$ (``coarser'')
\end{proposition}
\begin{proof} If $v,w \in \mathcal{R},$ Theorem \ref{qoringcomp}(2) yields that $K_v$ and $K_w$ are convex subrings of $K,$ without loss of generality $K_w \subseteq K_v.$ If equality holds, it is easy to verify that also $I_w = I_v,$ since both $K_v$ and $K_w$ are local. Thus, $v = w.$ If $K_w \subsetneq K_v,$ it follows again by the fact that $K_v$ and $K_w$ are local, that $I_v \subsetneq I_w,$ and therefore $v < w.$ 
\end{proof}

\begin{definition} The \textbf{rank of a quasi-ordered field} $(K,\preceq)$ is the order type of the totally ordered set $\mathcal{R}.$
\end{definition}

\noindent
For the following result we use that $(R,\preceq)$ is a quasi-ordered ring if and only if $(R/E_0,\preceq')$ is a quasi-ordered ring, where $\overline{x} \preceq' \overline{y} \Leftrightarrow x \preceq y$ (see \cite[Lemma 4.1]{sim}). Moreover, we exploit that the quasi-order $\preceq'$ uniquely extends to a quasi-order $\trianglelefteq$ on $K:=\mathrm{Quot}(R/E_0)$ via 
$$\frac{\overline{x}}{\overline{y}} \trianglelefteq \frac{\overline{a}}{\overline{b}} :\Leftrightarrow \overline{xyb^2} \preceq' \overline{aby^2}$$
(see \cite[Proposition 4.3]{sim}). If $v$ is a valuation on $R,$ let $v'$ denote the induced valuation on $R/E_0,$ and $\overline{v'}$ the extension of $v'$ to $K.$ We can now prove (see \cite[Lemma 4.1]{valente} for the same result in the ordered case):

\begin{lemma} \label{compequiv} Let $(R,\preceq)$ be a quasi-ordered ring and $v$ a valuation on $R$ with support $\mathfrak{q}_v = E_0.$ The following are equivalent:
\begin{enumerate}
\item $v$ is compatible with $\preceq.$ \vspace{1mm}

\item $v'$ is compatible with $\preceq'.$ \vspace{1mm}

\item $\overline{v'}$ is compatible with $\trianglelefteq.$
\end{enumerate}
\begin{proof} The equivalence of (1) and (2) easily follows from the definition of $\preceq'$ and $v'$, respectively. We conclude by showing that (2) and (3) are equivalent. It is clear that compatibility in $K$ reduces to compatibility in $R$ as it is a universal statement. For the contrary, let $0 \trianglelefteq \frac{x}{y} \trianglelefteq \frac{a}{b}.$ Then $0 \preceq xyb^2 \preceq aby^2.$ By compatibility with $v,$ we obtain that $v(aby^2) \leq v(xyb^2),$ i.e. that $v(ay) \leq v(xb).$ Thus, 
$\overline{v'}(\frac{a}{b}) \leq \overline{v'}(\frac{x}{y}).$
\end{proof}
\end{lemma}

\noindent
The equivalence of (1) and (3) from the previous lemma justifies to define:

\begin{definition} \label{qoringrank} The \textbf{rank of a quasi-ordered ring} $(R,\preceq)$ is the rank of the naturally associated quasi-ordered field $(\mathrm{Quot}(R/E_0),\trianglelefteq).$
\end{definition}


\noindent
For a further discussion on the rank of a quasi-ordered field we refer to \cite[p.403]{kuhl}. 

\section{The Baer-Krull Theorem}

\noindent
In the previous section we fixed a quasi-ordered ring $(R,\preceq)$ and characterized all the (local) Manis valuations on $R,$ that are compatible with $\preceq$ (see Theorem \ref{qoringcomp} and Lemma \ref{Iv1}). It is natural to ask what happens the other way round, i.e. if we fix a valued ring $(R,v)$ with $v$ Manis, can we describe all the quasi-orders on $R$ that are compatible with $v$? A positive answer is given by the Baer-Krull Theorem (see Theorem \ref{BaerKrullI}, respectively Corollary \ref{BaerKrullII}, Theorem \ref{BaerKrullIII}, Theorem \ref{BaerKrullIV}). Recall that if $\preceq$ is a $v$-compatible quasi-order on $R,$ then it gives rise to a quasi-order $\preceq'$ on the residue class domain $Rv := R_v/I_v.$ The said theorem establishes a connection between the $v$-compatible quasi-orders on $R$ with support $\textrm{supp}(v),$ and the quasi-orders on $Rv$ with support $\{0\}.$ After establishing the Baer-Krull Theorem for quasi-ordered rings, we deduce a version for ordered, respectively proper quasi-ordered, rings (see Corollary \ref{BaerKrullorings}, respectively Corollary \ref{BaerKrullqoringsII}). The first one yields a generalization of the classical Baer-Krull Theorem (see Theorem \ref{baerkrullfield}), while the latter gives rise to a characterization of all Manis valuations on $R$ that are finer than $v.$
\vspace{1mm}

\noindent
When one deals with quasi-ordered rings, this theorem becomes more complicated than in the ordered field case (see Theorem \ref{baerkrullfield}). Note that the map 
$\eta$ there is completely determined by the signs of the elements $\pi_i.$ If the quasi-order is an order, then all $\eta \in \{-1,1\}^I$ are realizeable and one gets a bijective correspondence as in Theorem \ref{baerkrullfield}. However, if the quasi-order is induced by some valuation, then all elements are non-negative, so the only $\eta$ possible is the constant map $\eta = 1.$ Therefore, when we consider quasi-ordered rings $(R,\preceq)$, the best we can hope for is that $\psi$ is an injective map such that the image of $\psi$ contains all possible tuples $(\eta_{\preceq},\preceq')$ as just described. Establishing such a result will be the aim for the rest of this paper. 

\begin{notation} \label{notation2} We use the following notation for the rest of this section:
\begin{enumerate}
 \item Let $R$ be a commutative ring with $1$ and $v \colon R \to \Gamma_v \cup \{\infty\}$ a Manis valuation on $R$ with support $\mathfrak{q}_v,$ valuation ring $R_v,$ valuation ideal $I_v,$ and residue class domain $Rv := R_v/I_v,$ just as in Notation \ref{notation}. Moreover, we define 
$\tilde{R}:= R \backslash \mathfrak{q}_v = v^{-1}(\Gamma_v).$ \vspace{1mm}

 \item We fix some $\mathbb{F}_2$-basis $\{\overline{\gamma_i}: i \in I\}$ of $\overline{\Gamma_v} = \Gamma_v/2\Gamma_v,$ and let 
$\{\pi_i: i \in I\} \subseteq \tilde{R}$ be such that $v(\pi_i) = \gamma_i.$ \vspace{1mm}

 \item Given a $v$-compatible quasi-order on $R,$ we denote by $\preceq'$ the induced quasi-order on $Rv$ (see Theorem \ref{qoringcomp}(3)). By $\eta_{\preceq}$ we denote the map $I \to \{-1,1\}$ defined by $\eta_{\preceq}(i) = 1$ if and only if $0 \preceq \pi_i.$
\end{enumerate}
\end{notation}

\noindent
Now we fix some tuple $(\eta^*,\preceq^*)$ from the disjoint union
\[
 \{-1,1\}^I \times \{\textrm{orders on} \, Rv \, \textrm{with support} \, \{0\}\} \sqcup \{1\}^I \times \{\textrm{p.q.o. on} \, Rv \, \textrm{with support} \, \{0\}\}
\]

\noindent
The main part of the proof of the Baer-Krull Theorem is to construct a quasi-order on $R$ that is mapped to $(\eta^*, \preceq^*)$ under the analogue of the map $\psi$ from Theorem \ref{baerkrullfield}. For that purpose we define a binary relation $\preceq$ on $R$ as a function of $\preceq^*$ and $\eta^*$ as follows: 
If $x,y \in \mathfrak{q}_v,$ declare $x \preceq y.$ Otherwise, if $x \in \tilde{R}$ or $y \in \tilde{R},$ consider 
$$\gamma := \gamma_{x,y} := \max\{-v(x),-v(y)\} \in \Gamma_v.$$ 
Write $\overline{\gamma} = \sum_{i \in I_{x,y}} \overline{\gamma_i}.$ Then $\gamma = \sum_{i \in I_{x,y}} \gamma_i + 2v(a_{x,y})$ for some $a_{x,y} \in \tilde{R},$ which is uniquely determined up to its value (i.e. instead of $a_{x,y}$ one may have chosen any other element $b_{x,y} \in R$ with $v(b_{x,y}) = v(a_{x,y})$). In what follows, we will just write $\sum_i, \prod_i$ and $a$ instead of $\sum_{i \in I_{x,y}}, \prod_{i \in I_{x,y}}$ and $a_{x,y},$ respectively, whenever $x$ and $y$ are clear from the context.

\begin{lemma} \label{residue} Let $x,y \in \tilde{R}.$ With the notation above, $x\prod_i \pi_ia^2,$ $y \prod_i \pi_i a^2 \in R_v.$ Moreover, $\overline{x\prod_i \pi_i a^2} = 0$ if and only if $v(x) > v(y).$
\begin{proof} Note that 
\begin{align*}
v\left(x\prod_i\pi_ia^2\right) &= v(x) + \sum\limits_i v(\pi_i) + 2v(a) = v(x) + \gamma \\ &= v(x) + \max\{-v(x),-v(y)\} \geq 0,
\end{align*}

\noindent
and likewise for $y\prod_i\pi_ia^2,$ so both are in $R_v.$ Moreover,
\begin{align*}
\overline{x\prod_i \pi_i a^2} = 0 \Leftrightarrow v(x) + \max\{-v(x),-v(y)\} > 0 \Leftrightarrow v(x)>v(y). 
\end{align*}
\end{proof}
\end{lemma}

\noindent
Particularly, we can take residues of both $x \prod_i \pi_ia^2$ and $y \prod_i \pi_i a^2.$ The moreover part of the statement will be of great importance in the proof of Main Lemma \ref{mainlemma}. For the latter, we also require the following two lemmas, which extend the statements from axiom (QR3), respectively (QR5). 

\begin{lemma} \label{QR3-} Let $(R,\preceq)$ be a quasi-ordered ring. If $x \preceq y$ and $z \preceq 0,$ then $yz \preceq xz.$
\begin{proof} As $E_0$ is an ideal, we may without loss of generality assume that $z \nsim 0.$ Moreover, note that if $y \sim 0,$ then $x,z \preceq 0,$ thus $0 \preceq -x,-z.$ It follows via (QR3) that $yz \sim 0 \preceq xz.$ So we may also assume that $y \notin E_0.$ From $x \preceq y$ and $z \preceq 0$ follows 
$-xz \preceq -yz.$ We claim that $yz \nsim -yz.$ Once this is shown, it follows from $-xz \preceq -yz$ that $yz-xz \preceq 0.$ The latter implies $yz \preceq xz.$ Indeed, either $x \nsim 0$ and therefore $xz \nsim 0$ ($E_0$ is a prime ideal), so that we can apply (QR4); or $x \sim 0,$ i.e. $xz \sim 0,$ and therefore 
$yz - xz \sim yz \preceq 0 \sim xz$ (see Remark \ref{rema}(1)). So assume for a contradiction that $yz \sim -yz.$ Lemma \ref{lemmaxxx} yields $0 \preceq yz, -yz.$ As $y \notin E_0,$ either $0 \prec y$ or $0 \prec -y.$ So via (QR5) it follows either from $0 \preceq yz$ (if $0 \prec y$) or from $0 \preceq -yz$ (if $0 \prec -y$) that $0 \preceq z.$ Hence $z \sim 0,$ a contradiction. 
\end{proof}
\end{lemma}

\begin{lemma} \label{QR5-} Let $(R,\preceq)$ be a quasi-ordered ring and $x,y,z \in R.$ If $xz \preceq yz$ and $z \prec 0,$ then $y \preceq x.$
\begin{proof} Assume for a contradiction $x \prec y.$ The previous lemma yields $yz \preceq xz.$ Hence $xz \sim yz.$ Lemma \ref{QR5} yields $x \sim y,$ a contradiction.
\end{proof}
\end{lemma}

\begin{mainlemma} \label{mainlemma}
With the notation from above, define for $x \in \tilde{R}$ or $y \in \tilde{R}$ that

 $$x \preceq y :\Leftrightarrow \left\{ \begin{array}{c c c c} \textrm{Either} & \overline{x\prod_i \pi_i a^2} \preceq^* \overline{y \prod_i \pi_i a^2} & \textrm{and} & \prod_i \eta^*(i) =1 \\
							    & & & \\
							    \textrm{or} & \overline{y\prod_i\pi_ia^2} \preceq^* \overline{x\prod_i\pi_i a^2} & \textrm{and} & \prod_i \eta^*(i) = -1.
                                      
                                     \end{array} \right.$$ \vspace{1mm}
																		
\noindent
Moreover, declare $x \preceq y$ for $x,y \in \mathfrak{q}_v.$ Then $\preceq$ defines a quasi-order on $R$ with support $E_0 = \mathfrak{q}_v.$
\end{mainlemma}
\begin{proof} The proof of the Main Lemma is extensive, however, the methods are widely the same. Notably, the moreover part from Lemma \ref{residue} is frequently exploited. We always use the notation from above. For the sake of convenience and uniformity, we treat $\preceq^*$ and $\eta^*$ as an arbitrary quasi-order on $Rv$ with support $\{0\},$ respectively an arbitrary map from $I$ to $\{-1,1\},$ for as long as possible. In fact, the distinction whether $\preceq^*$ is an order or induced by a valuation (in which case the map $\eta^*$ is trivial) is only necessary at some points when we verify axiom (QR4). \vspace{1mm}

\noindent
First of all we show that $\preceq$ is well-defined. Recall that $a \in \tilde{R}$ was only determined up to its value. So let $b \in \tilde{R}$ with $v(a) = v(b),$ and suppose that $$\overline{x\prod_i \pi_i a^2} \preceq^* \overline{y \prod_i\pi_ia^2}.$$ As $v$ is Manis, there exists some $z \in \tilde{R}$ with $v(z) = -v(b),$ so $v(bz) = 0,$ i.e. $\overline{bz} \neq 0.$ Particularly, $0 \prec^* \overline{bz}^2.$ With axiom (QR3) follows, after rearranging, that
$$\overline{x \prod_i\pi_i b^2} \ \overline{az}^2 \preceq^* \overline{y \prod_i\pi_i b^2} \ \overline{az}^2.$$ We conclude by eliminating $0 \prec^* \overline{az}^2$ via axiom (QR5).  \vspace{1mm}

\noindent
Clearly, $\preceq$ is reflexive and total. At next we prove transitivity. So let $x \preceq y$ and $y \preceq z,$ without loss of generality $x \in \tilde{R}$ or $z \in \tilde{R}.$ The proof is done by distinguishing four cases. First of all assume that $v(p) = v(q) \leq v(r)$ with $p,q,r \in \{x,y,z\}$ pairwise distinct. Then $\gamma_{x,y} = \gamma_{x,z} = \gamma_{y,z} \in \Gamma_v$ all coincide, so $I_{x,y} = I_{x,z} = I_{y,z}$ and $a_{x,y} = a_{x,z} = a_{y,z}.$ Hence, transitivity of $\preceq$ follows immediately by transitivity of $\preceq^*.$ It remains to verify the cases where there is a unique smallest element among $v(x),v(y)$ and $v(z).$ First suppose that $v(x)<v(y),v(z).$ Then $\gamma_{x,y} = -v(x) = \gamma_{x,z},$ i.e. $I_{x,y} = I_{x,z}$ and $a_{x,y} = a_{x,z}.$ We do the case $\prod_{i \in I_{x,y}} \eta^*(i) = -1,$ the case $\prod_{i \in I_{x,y}} \eta^*(i) = 1$ is proven likewise. From $x \preceq y$ and $v(x) < v(y)$ then follows 
$$\overline{y\prod_{i \in I_{x,y}} \pi_i a_{x,y}^2} = 0 \preceq^* \overline{x \prod_{i \in I_{x,y}} \pi_i a_{x,y}^2}$$ (see Lemma \ref{residue}). Now $v(x) < v(z)$ and Lemma \ref{residue} imply that $$\overline{z \prod_{i \in I_{x,y}} \pi_i a_{x,y}^2} = 0.$$ Therefore, $x \preceq z.$ 
Next, suppose that $v(y) < v(x),v(z).$ Then $\gamma_{x,y} = -v(y) = \gamma_{y,z},$ i.e. $I_{x,y} = I_{y,z}$ and 
$a_{x,y}=a_{y,z}.$ Again, we only do the case $\prod_{i \in I_{x,y}} \eta^*(i) = -1.$ From $v(y) < v(x)$ and $x \preceq y$ follows 
$$\overline{y\prod_{i \in I_{x,y}}\pi_ia_{x,y}^2} \preceq^* \overline{x\prod_{i \in I_{x,y}}\pi_ia_{x,y}^2} = 0.$$ Likewise, $v(y) < v(z)$ and $y \preceq z$ implies $$0 \preceq^* \overline{y\prod_{i \in I_{x,y}}\pi_ia_{x,y}^2}.$$ Since the support of $\preceq^*$ is trivial, it follows $$\overline{y\prod_{i \in I_{x,y}}\pi_ia_{x,y}^2} = 0.$$ On the other hand, $v(y) < v(x),v(z)$ yields via Lemma \ref{residue} that $$\overline{y\prod_{i \in I_{x,y}}\pi_ia_{x,y}^2} \neq 0,$$ a contradiction. The case $v(z) < v(x),v(y)$ is proven like the case where $v(x)$ is the unique smallest value. \vspace{1mm}

\noindent
Next, we establish that the support of $\preceq$ is $\mathfrak{q}_v.$ Assume there is some $x \in E_0$ with $x \notin \mathfrak{q}_v.$ Then $\overline{x\prod_i\pi_ia^2} \sim 0.$ As the support of $\preceq^*$ is $\{0\},$ this yields $x \prod_i\pi_ia^2 \in I_v.$  However, as $v(x) < v(0) = \infty,$ this contradicts Lemma \ref{residue}. We obtain that $E_0 \subseteq \mathfrak{q}_v.$ The other implication follows immediately from the definition of $\preceq.$ \vspace{1mm}

\noindent
It remains to verify the axioms (QR1) - (QR5) and compatibility with $v.$ For the proof of \textbf{(QR1)} assume for a contradiction that $1 \preceq 0.$ Note that $\gamma_{0,1} = 0,$ so $I = \emptyset,$ and $\prod_i\eta^*(i) = 1.$ It follows from $1 \preceq 0$ that $\overline{a}^2 \preceq^* 0$ for some $a \in R$ with $v(a) = 0,$ i.e. $\overline{a} \neq 0.$ This contradicts the facts that squares are non-negative and that the support of $\preceq^*$ is trivial. \vspace{1mm}

\noindent
For \textbf{(QR2)} is nothing to show by Remark \ref{rema}(3). Next, we verify \textbf{(QR3)}, i.e. we show that $x \preceq y$ and $0 \preceq z$ implies 
$xz \preceq yz.$ By the definition of $\preceq$ and the fact that $\mathfrak{q}_v$ is an ideal, we may without loss of generality assume that 
$z \notin \mathfrak{q}_v,$ and that not both $x$ and $y$ are in $\mathfrak{q}_v.$ Further note that
\begin{align*}
 \gamma_{xz,yz} &= \max\{-v(xz),-v(yz)\} = \max\{-v(z),-v(0)\} + \max\{-v(x),-v(y)\} \\ &= \gamma_{0,z} + \gamma_{x,y} \in \Gamma_v.
\end{align*}
Hence, $I_{xz,yz}$ is the (without loss of generality) disjoint union of $I_{x,y}$ and $I_{0,z},$ which implies $$\prod_{i \in I_{xz,yz}}\eta^*(i) = \prod_{i \in I_{x,y}} \eta*(i) \cdot \prod_{i \in I_{0,z}} \eta^*(i).$$ Moreover, $a_{xz,yz} = a_{x,y}a_{0,z}.$ \vspace{1mm}

\noindent
First consider the case $\prod_{i \in I_{0,z}}\eta^*(i) = 1.$ This yields $0 \preceq^* \overline{z\prod_{i \in I_{0,z}}\pi_ia_{0,z}^2}.$ Now let 
$\prod_{i \in I_{x,y}}\eta^*(i) = -1 = \prod_{i \in I_{xz,yz}}\eta^*(i).$ Then $$\overline{y\prod_{i \in I_{x,y}}\pi_ia_{x,y}^2} \preceq^* \overline{x \prod_{i \in I_{x,y}} \pi_ia_{x,y}^2}.$$ Applying (QR3) yields $$\overline{yz\prod_{i \in I_{xz,yz}}\pi_ia_{x,y}^2a_{0,z}^2} \preceq^* \overline{xz\prod_{i \in I_{xz,yz}}\pi_ia_{x,y}^2a_{0,z}^2}.$$ Hence, $xz \preceq yz.$ The case $\prod_{i \in I_{x,y}} \eta^*(i) = 1 = \prod_{i \in I_{xz,yz}}\eta^*(i)$ is proven analogously. The proof for 
$\prod_{i \in I_{0,z}}\eta^*(i) = -1$ is also almost the same; we just apply Lemma \ref{QR3-} instead of axiom (QR3). \vspace{1mm}

\noindent
The proof of axiom \textbf{(QR4)} is divided into five subcases. First suppose that $v(x) < v(z)$ or $v(y) < v(z).$ Either way, $\gamma_{x,y} = \gamma_{x+z,y+z}.$ Moreover, in both cases $\overline{z\prod_{i \in I_{x,y}}\pi_ia_{x,y}^2} = 0.$ From this observation, the claim follows immediately. Further note that if $\preceq^*$ is an order and $x \prec y,$ we obtain $x+z \prec y+z,$ because orders preserve strict inequalities under addition. We will exploit this fact below. \vspace{1mm}

\noindent  
If $v(z)<v(x),v(y),$ then $\gamma_{y,z} = \gamma_{x+z,y+z}.$ Note that $$\overline{x\prod_{i \in I_{y,z}}\pi_ia_{y,z}^2} = 0 = \overline{y\prod_{i \in I_{y,z}}\pi_ia_{y,z}^2}$$ by Lemma \ref{residue}. It is easy to see that then $x+z \preceq y+z.$ Suppose now that $v(x) = v(z) < v(y).$ Then $\gamma_{x,y} = \gamma_{y,z} = \gamma_{x+z,y+z}.$ From $v(x)<v(y)$ follows that $\overline{y\prod_{i \in I_{x,y}}\pi_ia_{x,y}^2} = 0.$ We distinguish two subcases. 
If $\prod_{i \in I_{x,y}}\eta^*(i) = 1,$ then $x \preceq y$ yields $$\overline{x\prod_{i \in I_{x,y}}\pi_ia_{x,y}^2} \preceq^* \overline{y\prod_{i \in I_{x,y}}\pi_ia_{x,y}^2} = 0.$$ As $z \nsim y,$ one may add $\overline{z\prod_{i \in I_{x,y}}\pi_ia_{x,y}^2}$ on both sides. This concludes the present subcase. On the other hand, if $\prod_{i \in I_{x,y}}\eta^*(i) = -1,$ then $\eta^* \neq 1,$ so $\preceq^*$ is an order. Thus, one may add $\overline{z\prod_{i \in I_{x,y}}\pi_ia_{x,y}^2}$ on both sides as well. Next, suppose that $v(y) = v(z) < v(x).$ Then $\gamma_{x,y} = \gamma_{x+z,y+z}.$ It holds $\overline{x\prod_{i \in I_{x,y}}\pi_ia_{x,y}^2} = 0.$ If $\prod_{i \in I_{x,y}}\eta^*(i) = 1,$ then $x \preceq y$ yields $0 \preceq^* \overline{y\prod_{i \in I_{x,y}}\pi_ia_{x,y}^2},$ and one may conclude by adding $\overline{z\prod_{i \in I_{x,y}}\pi_ia_{x,y}^2}$ on both sides. Analogously, if $\prod_{i \in I_{x,y}}\eta^*(i) = -1,$ one may also add $\overline{z\prod_{i \in I_{x,y}}\pi_ia_{x,y}^2}$ on both sides, as the support of $\preceq^*$ is zero. Finally, suppose that $v(x) = v(y) = v(z).$ If they are all in $\mathfrak{q}_v,$ then also $x+z,y+z \in \mathfrak{q}_v,$ and therefore $x+z \preceq y+z.$ So suppose that $v(x) = v(y) = v(z) \in \Gamma_v.$ It holds
\[
 \gamma_{x+z,y+z} = \max\{-v(x+z),-v(y+z)\} \leq -v(z).
\]
First suppose that equality holds. Then $\max\{-v(x+z),-v(y+z)\} = -v(z),$ i.e. all $\gamma's$ coincide. If $\prod_i\eta^*(i) = 1,$ the claim follows immediately from (QR4) and the fact that $y \nsim z$ by simply adding $\overline{z\prod_i\pi_ia^2}$ to both sides of the inequality $\overline{x\prod_i\pi_ia^2} \preceq^* \overline{y\prod_i\pi_ia^2}.$ If $\prod_i\eta^*(i) = -1,$ then $\preceq^*$ must be an order and we may simply add $\overline{z\prod_i\pi_ia^2}$ on both sides anyway. \vspace{1mm}

\noindent
Last but not least assume that $<$ holds, i.e. $\max\{-v(x+z),-v(y+z)\} < -v(z).$ Then $v(z) < (x+z),v(y+z).$ By Lemma 4.2, $\overline{p\prod_{i \in I_{x,y}}\pi_ia_{x,y}^2} \neq 0$ for $p \in \{x,y,z\},$ whereas
\[
 \overline{(x+z)\prod_{i \in I_{x,y}}\pi_ia_{x,y}^2} = 0 = \overline{(y+z)\prod_{i \in I_{x,y}}\pi_ia_{x,y}^2}.
\]
Therefore, $$\overline{x\prod_{i \in I_{x,y}}\pi_ia_{x,y}^2} = \overline{y\prod_{i \in I_{x,y}}\pi_ia_{x,y}^2} = -\overline{z\prod_{i \in I_{x,y}}\pi_ia_{x,y}^2}.$$ Particularly, we may assume that $\preceq^*$ is an order, since in the proper quasi-ordered case 
$$\overline{y\prod_{i \in I_{x,y}}\pi_ia_{x,y}^2} \sim -\overline{y\prod_{i \in I_{x,y}}\pi_ia_{x,y}^2} = \overline{z\prod_{i \in I_{x,y}}\pi_ia_{x,y}^2},$$ contradicting the assumption $y \nsim z.$ We claim that
$x+z \sim 0 \sim y+c,$ which clearly implies $x+z \preceq y+z.$ Assume for a contradiction that $x+z \nsim 0.$ If $x+z \prec 0,$ it follows from the case ``$v(x)<v(z)$'' (where $x+z$ plays the role of $x,$ $0$ the one of $y$ and
$-z$ the one of $z;$ recall that $v(x+z) < v(z)$) above and the fact that $\preceq^*$ is an order, that $x \prec -z,$ contradicting $x \sim -z.$ Likewise, if $0 \prec x+z,$ it follows from the case ``$v(y)<v(z)$'' that $-z \prec x,$
again a contradiction. Therefore $x+z \sim 0.$ The same reasoning shows that $y+z \sim 0$ as well. \vspace{1mm}

\noindent
Finally, we prove axiom \textbf{(QR5)}. Suppose that $xz \preceq yz$ and $0 \prec z.$ Clearly $z \in \tilde{R},$ as $0 \nsim z.$ Moreover, let without loss of generality $x \in \tilde{R}$ or $y \in \tilde{R}.$ Note that $\gamma_{xz,yz} = \gamma_{x,y} + \gamma_{z,0},$ and 
$\prod_{i \in I_{xz,yz}}\eta^*(i) = \prod_{i \in I_{x,y}}\eta^*(i)\prod_{i \in I_{0,z}}\eta^*(i),$ and $a_{xz,yz} = a_{0,z}a_{x,y},$ as in the proof of (QR3) above. \vspace{1mm}

\noindent
First let $\prod_{i \in I_{0,z}} \eta^*(i) = 1.$ Then $0 \prec^* \overline{z\prod_{i \in I_{0,z}}\pi_ia_{0,z}^2}.$ If $\prod_{i \in I_{xz,yz}}\eta^*(i) = -1,$ also $\prod_{i \in I_{x,y}}\eta^*(i) = -1.$ We obtain
$$\overline{yz\prod_{i \in I_{xz,yz}}\pi_ia_{x,y}^2a_{0,z}^2} \preceq^* \overline{xz\prod_{i \in I_{xz,yz}}\pi_ia_{x,y}^2a_{0,z}^2}.$$ Eliminating $\overline{z\prod_{i \in I_{0,z}}\pi_ia_{0,z}^2}$ via (QR5) yields $$\overline{y\prod_{i \in I_{x,y}}\pi_ia_{x,y}^2} \preceq^* \overline{x\prod_{i \in I_{x,y}}\pi_ia_{x,y}^2},$$ and
therefore $x \preceq y.$ If $\prod_{i \in I_{xz,yz}}\eta^*(i) = 1,$ then $\prod_{i \in I_{x,y}}\eta^*(i) = 1,$ and the proof is likewise. \vspace{1mm}

\noindent
Now let $\prod_{i \in I_{0,z}}\eta^*(i) = -1.$ Then $\overline{z\prod_{i \in I_{0,z}}\pi_ia_{0,z}^2} \prec^* 0.$ If $\prod_{i \in I_{xz,yz}}\eta^*(i) = -1,$ then $\prod_{i\in I_{x,y}}\eta^*(i) = 1.$ It follows
$$\overline{yz\prod_{x \in I_{xz,yz}}\pi_ia_{x,y}^2a_{0,z}^2} \preceq^* \overline{xz\prod_{i \in I_{xz,yz}}\pi_ia_{x,y}^2a_{0,z}^2}.$$ Applying Lemma \ref{QR5-} yields 
$\overline{x\prod_{i \in I_{x,y}}\pi_ia_{x,y}^2} \preceq^* \overline{y\prod_{i \in I_{x,y}}\pi_ia_{x,y}^2}.$ Therefore, $x \preceq y.$ The case
$\prod_{i \in I_{xz,yz}}\eta^*(i) = 1$ is analogue. \vspace{1mm}

\noindent
We conclude by showing that $\preceq$ is $v$-compatible. Suppose $0 \preceq x \preceq y$ but $v(x) < v(y)$ for some $x,y \in R.$ Note that $\gamma_{0,x} = -v(x) = \gamma_{x,y}.$ If $\prod_i\eta^*(i) = -1,$ then $0 \preceq x$ yields $$\overline{x\prod_{i \in I_{0,x}}\pi_i a_{0,x}^2} \prec^* 0 = \overline{y \prod_{i \in I_{0,x}} \pi_ia_{0,x}^2},$$ i.e. $y \prec x,$ a contradiction. The same argument works for $\prod_{i \in I_{0,x}}\eta^*(i) = 1.$
\end{proof}

\begin{remark} \label{baereasy} The quasi-order $\preceq$ from the Main Lemma becomes very simple in the case where $x \in U_v$ and $y \in R_v$ (or vice versa). Note that then $\gamma_{x,y} = 0.$ This implies $I = \emptyset.$ Hence, $\prod_i \eta^*(i) = 1.$ Moreover, the element $a$ satisfies $v(a) = 0,$ so by well-definedness of $\preceq$ we may simply choose $a = 1.$ Therefore $x \preceq y \Leftrightarrow \overline{x} \preceq^* \overline{y}.$ 
\end{remark}

\noindent
For the proof of the Baer-Krull Theorem we require two more lemmas. They will be used to compare the ``size'' of two quasi-orders on $R.$ 

\begin{lemma} \label{lemmainj1} Let $(R,\preceq)$ be a quasi-ordered ring and $x \in R.$ Then $E_0 + \{x\} \subseteq E_x.$
\begin{proof} For $x \in E_0$ there is nothing to show. So let $y \in R \backslash E_0$ such that $y = c + x$ for some $c \in E_0.$ Remark \ref{rema}(1) yields $c+x \sim x,$ so $y \in E_x.$
\end{proof}
\end{lemma}

\begin{lemma} \label{lemmainj2} Let $(R,\preceq)$ be a quasi-ordered ring and $x \in R.$ If $E_0 + \{x\} \subsetneq E_x,$ then $E_x = -E_x.$
\begin{proof} Let $z \in E_x$ be arbitrary and $y \in E_x \backslash (E_0 + \{x\}).$ We will show that $-y \in E_x.$ From $z \sim x \sim y$ and Corollary 
\ref{corpres} then follows $-z \sim -y \sim x,$ i.e. also $-z \in E_x,$ what proves that $E_x = -E_x.$ \vspace{1mm}

\noindent
The proof that $-y \in E_x$ is like in \cite[p.208]{fakh}. Assume for a contradiction that $-y \notin E_x.$ Then $y \preceq x \nsim -y,$ thus $0 \preceq x-y.$ Likewise, it follows from $x \preceq y \nsim -y$ that $x-y \preceq 0.$ Therefore, $x-y \in E_0,$ i.e. $y \in E_0 + \{x\},$ a contradiction. Hence, $-y \in E_x,$ i.e. $E_x = -E_x.$
\end{proof}
 
\end{lemma}

\begin{notation} For a prime ideal $\mathfrak{p}$ of $R$ denote by $\mathcal{X}_{\mathfrak{p}}(R)$ the set of all quasi-orders on $R$ with support $\mathfrak{p}.$ Analogously, denote by $\mathcal{X}_{o,\mathfrak{p}}(R)$ (respectively $\mathcal{X}_{p,\mathfrak{p}}(R)$) the set of all orders (respectively proper quasi-orders) on $R$ with support $\mathfrak{p}.$ If $R$ is a field, we omit the index $\mathfrak{p}$ for obvious reasons.
\end{notation}

\noindent
In the Baer-Krull Theorem we demand that the support of the quasi-orders coincides with the support of our valuation. Note that it actually suffices to demand $\textrm{supp}(v) \subseteq \textrm{supp}(\preceq),$ the other implication being implied as follows: If $x \in \textrm{supp}(\preceq),$ then $0 \preceq x \preceq 0,$ so compatibility yields $v(x) = \infty \in \mathfrak{q}_v.$

\begin{theorem} \label{BaerKrullI} $\mathrm{(}$Baer-Krull Theorem for quasi-ordered rings I $\mathrm{)}$ \\
Let $R$ be a commutative ring with $1$ and $v$ a Manis valuation on $R.$ Then
\begin{align*}
\psi \colon \{\preceq \in \mathcal{X}_{\mathfrak{q}_v}(R)\colon \preceq \textrm{is } v\textrm{-compatible} & \} \to \{-1,1\}^I \times \mathcal{X}_{\{0\}}(Rv), \\ &\preceq \; \mapsto \; (\eta_{\preceq},\preceq')
\end{align*}
is a well-defined map such that $\psi \restriction \psi^{-1}(\mathcal{A}): \psi^{-1}(\mathcal{A}) \to \mathcal{A}$ is a bijection, where 
$\mathcal{A} := \{-1,1\}^I \times \mathcal{X}_{o,\{0\}}(Rv) \sqcup \{1\}^I \times \mathcal{X}_{p,\{0\}}(Rv).$
\end{theorem}
\begin{proof}
By Theorem \ref{qoringcomp}(3) the map $\psi$ is well-defined. Next, let $(\eta^*, \preceq^*) \in \mathcal{A}$ be arbitrary. We prove that $\psi$ maps the quasi-order $\preceq$ constructed in the Main Lemma to the tuple $(\eta^*,\preceq^*).$ First we verify that $\eta_{\preceq} = \eta^*.$ To compare $\pi_i$ and $0$ w.r.t. $\preceq,$ let $\gamma := \max\{-v(\pi_i),-v(0)\} = -\gamma_i,$ i.e. $\gamma = v(\pi_ia^2)$ for some $a \in \tilde{R}.$ Hence, we have to consider $0$ and 
$\pi_i \pi_i a^2 = (\pi_ia)^2,$ and to distinguish whether $\eta^*(i)$ equals $1$ or $-1.$ Note that $0 \prec^* \overline{\pi_ia}^2,$ as it is a square and 
$\preceq^*$ has trivial support. From this observation we obtain $$\eta_{\preceq}(i) = 1 \Leftrightarrow 0 \preceq \pi_i \Leftrightarrow \eta^*(i) = 1,$$ and therefore $\eta_{\preceq} = \eta^*.$ \vspace{1mm}

\noindent
Next, we prove that $\preceq' = \preceq^*.$ Assume without loss of generality that not both $x,y \in I_v.$ Then also $x+c$ and $y+d$ are not both in $I_v$ for all $c,d \in I_v.$ It follows from Remark \ref{baereasy} that $x+c \preceq y+d \Leftrightarrow \overline{x+c} \preceq^* \overline{y+d}.$ Thus,
\begin{align*}
 \overline{x} \preceq' \overline{y} &\Leftrightarrow \exists c_1,c_2 \in I_v: x+c_1 \preceq y+c_2 \\& \Leftrightarrow \exists c_1,c_2 \in I_v: \overline{x+c_1} \preceq^* \overline{y+c_2}
 \\& \Leftrightarrow \overline{x} \preceq^* \overline{y}, 
\end{align*}
where the first equivalence just uses the definition of $\preceq'.$ \vspace{1mm}

\noindent
We conclude by showing that $\psi \restriction \psi^{-1}(\mathcal{A})$ is injective. Let $\preceq_1 \in \psi^{-1}(\mathcal{A})$ be arbitrary, and denote by $\preceq_2$ the quasi-order on $R$ defined by $\eta_{\preceq_1}$and $\preceq_1'$ (see Main Lemma). We prove that $\preceq_1 = \preceq_2.$ First of all we claim that $\preceq_1 \subseteq \preceq_2.$ So let $x,y \in R.$ Since $\preceq_1$ and $\preceq_2$ have both support $\mathfrak{q}_v,$ we may without loss of generality assume that 
$x \notin \mathfrak{q}_v$ or $y \notin \mathfrak{q}_v.$ Let $I, \pi_i$ and $a$ be as in the definition of the quasi-order $\preceq_2.$ 
First suppose that $\prod_i \eta_{\preceq_1}(i) = -1,$ i.e. $\prod_i\pi_ia^2 \prec_1 0.$ With Lemma \ref{QR3-} and Lemma \ref{QR5-}, we obtain
\begin{align*}
     x \preceq_1 y &\Leftrightarrow y \prod_i \pi_i a^2 \preceq_1 x \prod_i \pi_i a^2 \\ & \Rightarrow \overline{y\prod_i\pi_ia^2} \preceq_1' \overline{x\prod_i\pi_ia^2} \\ & \Leftrightarrow x \preceq_2 y. 
\end{align*}
Likewise, if $\prod_i \eta_{\preceq_1}(i) = 1,$ we just apply (QR3) instead of Lemma \ref{QR3-} and (QR5) instead of Lemma \ref{QR5-} to get the same result. Thus, $\preceq_1 \subseteq \preceq_2.$ For the rest of the proof we distinguish the cases $-1 \nsim_{\preceq_2} 1$ and $-1 \sim_{\preceq_2} 1.$ \vspace{1mm}

\noindent
If $-1 \nsim_{\preceq_2} 1,$ then Lemma \ref{QR5} yields $-x \nsim_{\preceq_2} x$ for all $x \in \tilde{R},$ so $E_{x,\preceq_2} \neq -E_{x, \preceq_2}$ for all such $x.$ From Lemma \ref{lemmainj1} and Lemma \ref{lemmainj2} follows $E_{x,\preceq_2} = \mathfrak{q}_v + \{x\}$ for all $x \in R.$ So Lemma \ref{lemmainj1} yields that $\preceq_2$ is the smallest quasi-order with support $\mathfrak{q}_v$ possible. Therefore, $\preceq_1 \subseteq \preceq_2$ implies equality, as desired. So suppose for the rest of this proof that $-1 \sim_{\preceq_2} 1.$ We distinguish the subcases $v(x) \neq v(y)$ and $v(x) = v(y).$ \vspace{1mm}

\noindent
If $v(x) \neq v(y),$ then Lemma \ref{residue} states $\overline{x \prod_i \pi_ia^2} \neq 0$ and $\overline{y \prod_i \pi_ia^2} = 0,$ or vice versa. We show $\preceq_1 = \preceq_2$ by proving that the only $\Rightarrow$ above is also an equivalence. First suppose that $\overline{y\prod_i\pi_ia^2} = 0.$ Assume for a contradiction that $$0 = \overline{y\prod_i\pi_ia^2} \preceq_1' \overline{x\prod_i\pi_ia^2} \textrm{ but } x \prod_i \pi_ia^2 \prec_1 y\prod_i\pi_ia^2.$$ Then we find some 
$c_1, c_2 \in I_v$ such that $c_1 \preceq_1 x \prod_i \pi_ia^2 + c_2.$ With Lemma \ref{convexity}(1) follows 
$c_1-c_2 \preceq_1 x\prod_i\pi_ia^2 \prec_1 y \prod_i \pi_ia^2,$ thus convexity of $I_v$ yields $x\prod_i\pi_ia^2 \in I_v,$ contradicting 
$\overline{x\prod_i\pi_ia^2} \neq 0.$ Now suppose that $\overline{x\prod_i\pi_ia^2} = 0.$ Then we obtain that 
$$y\prod_i\pi_ia^2 + c \preceq_1 x \prod_i \pi_ia^2 \prec_1 y \prod_i\pi_ia^2,$$ and taking residues yields that $\overline{y\prod_i\pi_ia^2} = 0,$ since the support of $\preceq'$ is trivial, a contradiction. \vspace{1mm}

\noindent
So finally suppose that $v(x) = v(y),$ and assume for a contradiction that $x \sim_{\preceq_2} y,$ but $x \prec_1 y.$ Choose $a \in \tilde{R}$ such that $0 \prec_1 a$ (and hence $0 \prec_2 a)$ and $v(a) = -v(x).$ Note that $ax \prec_1 ay$ if and only if $x \prec_1 y$ (by (QR5) and (QR3)), and also $ax \sim_{\preceq_2} ay$ if and only if $x \sim_{\preceq_2} y$ (Lemma \ref{QR5} and Corollary \ref{corpres}). So we may replace $x$ and $y$ with $ax$ and $ay.$ In other words, we may without loss of generality assume that $v(x) = v(y) = 0.$ It holds $y \preceq_2 x.$ So by definition of $\preceq_2$ and the fact that $v(x) = v(y) = 0,$ we get that $\overline{y} \preceq_1' \overline{x}$ (see Remark \ref{baereasy}). Thus, there exist some $c_1,c_2 \in I_v$ such that $y+c_1 \preceq_1 x+c_2,$ respectively, $y \preceq_1 x+c$ for $c:= c_2-c_1$ (see Lemma \ref{convexity}(1)). Recall that $-1 \sim_{\preceq_2} 1.$ But then also $-1 \sim_{\preceq_1} 1.$ Otherwise $-1 \preceq_1 0,$ but 
$-1 \npreceq_2 0,$ contradicting the fact that $\preceq_1 \subseteq \preceq_2.$ Therefore, Corollary \ref{corpres} and Lemma \ref{lemmaxxx} yield that all elements in $R$ are non-negative with respect to $\sim_1.$ Particularly, $0 \prec_1 -1.$ So Lemma \ref{lemmaineq} implies $y \preceq_1 x+c \preceq_1 \max\{x,c\} \prec_1 y,$ a contradiction (note that $y \preceq c$ would contradict the convexity of $I_v,$ as $0 \prec_1 y$). This finishes the proof of the Baer-Krull Theorem.
\end{proof}

\noindent
Note that for the sake of uniformity, we avoided the dichotomy from \cite[Theorem 3.8]{sim}, stating that every quasi-ordered ring is either an ordered or else a valued ring, throughout the entire paper. Taking this theorem into consideration, the Baer-Krull Theorem simplifies as follows:

\begin{corollary} \label{BaerKrullII} $\mathrm{(}$Baer-Krull Theorem for quasi-ordered rings II $\mathrm{)}$ \\
Let $R$ be a commutative ring with $1$ and $v$ a Manis valuation on $R.$ Then the map
\begin{align*}
\psi \colon \{\preceq \; \in \mathcal{X}_{\mathfrak{q}_v}(R)\colon \preceq \textrm{is } v&\textrm{-compatible}\} \to \mathcal{A}, \\& \preceq \; \mapsto \; (\eta_{\preceq},\preceq')
\end{align*}
is a bijection, where $\mathcal{A}$ is defined as in Theorem \ref{BaerKrullI}.
\end{corollary}
\begin{proof} \cite[Theorem 3.8]{sim} and Remark \ref{qoringcompapp}(4) yield that $\psi^{-1}(\mathcal{A})$ coincides with the domain of $\psi.$ The statement follows now immediately from the previous theorem.
\end{proof}

\noindent
We continue our discussion of the Baer-Krull Theorem by weakening the assumption that $v$ is Manis. So let $(R,v)$ be an arbitrary valued ring with support $\mathfrak{p}$ and value group $\Gamma := \Gamma_v.$ Note that the Manis property is not necessary to choose an $\mathbb{F}_2$-basis of $\overline{\Gamma} = \Gamma/2\Gamma$ with preimages in $R,$ as for any $\gamma \in \Gamma$ either $\gamma \in v(\tilde{R}),$ or $-\gamma \in v(\tilde{R}),$ or both. Let the $\pi_i, \gamma_i$ etc. be as above. Furthermore, let $\nu$ denote the unique extension from $v$ to $K:= \textrm{Quot}(R/\mathfrak{p}).$ Then Corollary \ref{BaerKrullII} yields a bijective correspondence

\begin{align*}
\psi \colon \{\trianglelefteq \; \in \mathcal{X}(K)\colon \trianglelefteq \textrm{is } \nu&\textrm{-compatible}\} \to \mathcal{A}, \\& \trianglelefteq \; \mapsto \; (\eta_{\trianglelefteq},\trianglelefteq')
\end{align*}
where $\mathcal{A} := \{-1,1\}^I \times \mathcal{X}(K\nu) \sqcup \{1\}^I \times \mathcal{X}(K\nu).$ \\

\noindent
From \cite[Proposition 2.7]{sim} we know that any quasi-order $\preceq$ on $R$ with support $\mathfrak{p}$ uniquely extends to a quasi-order $\trianglelefteq$ on $K:= \textrm{Quot}(R/\mathfrak{p}).$ Moreover, from Lemma \ref{compequiv} follows that $\preceq$ is $v$-compatible if and only if $\trianglelefteq$ is $\nu$-compatible. Hence, there is a bijective correspondence
\begin{align*}
\lambda: \{\preceq \; \in \mathcal{X}_{\mathfrak{p}}(R)\colon \preceq \textrm{is } v&\textrm{-compatible}\} \to \{\trianglelefteq \; \in \mathcal{X}(K)\colon \trianglelefteq \textrm{is } \nu\textrm{-compatible}\}, \\& \preceq \; \mapsto \; \trianglelefteq. 
\end{align*}

\noindent
Considering the composition $\psi \circ \lambda$ yields:

\begin{theorem} \label{BaerKrullIII} (Baer-Krull Theorem for quasi-ordered rings III) \\
Let $R$ be a commutative ring with $1$ and $v$ a valuation on $R.$ Then the map
\begin{align*}
\psi \circ \lambda \colon \{\preceq \; \in \mathcal{X}_{\mathfrak{p}}(R)\colon \preceq \textrm{is } v\textrm{-compatible}&\} \to \mathcal{A}, \\ &\; \preceq \; \mapsto \; (\eta_{\trianglelefteq},\trianglelefteq')
\end{align*}
is a bijection, where $\mathcal{A} := \{-1,1\}^I \times \mathcal{X}(K\nu) \sqcup \{1\}^I \times \mathcal{X}(K\nu).$
\end{theorem}

\noindent
In the last step we want the co-domain to go back from $K\nu$ to $Rv$ again. Note that if $v$ is a valuation on $R$ with support $\mathfrak{p},$ then $Rv$ is a domain. So we can consider $L :=\textrm{Quot}(Rv).$ We can also take the extension $\nu$ from $v$ to $K= \textrm{Quot}(R/\mathfrak{p}),$ and then consider the residue class field $K\nu = K_{\nu}/I_{\nu}.$ 

\begin{lemma} \label{fields} Let $v$ be a valuation on $R$ with support $\mathfrak{p},$ and let $\nu$ denote the unique extension from $v$ to $K:= \textrm{Quot}(R/\mathfrak{p}).$ Then $L:= \textrm{Quot}(Rv)$ is (isomorphic to) a subfield of $K\nu.$
\end{lemma}
\begin{proof} We consider the canonical map
\[
\phi: R_v \to K\nu, \; x \mapsto \frac{x+\mathfrak{p}}{1+\mathfrak{p}} + I_{\nu}.
\] 
Note that
\[
x \in \textrm{ker}(\phi) \Leftrightarrow \frac{x+\mathfrak{p}}{1+\mathfrak{p}} \in I_{\nu} \Leftrightarrow x+\mathfrak{p} \in I_{v'} \Leftrightarrow x \in I_v,
\]
where $v'$ denotes the valuation on $R/\mathfrak{q}_v$ defined by $v'(\overline{x}) = v(x)$ (see \cite[Lemma 2.5]{sim}). So the homomorphism theorem yields that $K\nu$ is a field containing the domain $R_v/I_v = Rv.$ Hence, it also contains its quotient field $L.$
\end{proof}

\begin{definition} (\cite[p. 975]{valente}) Let $R$ be a commutative ring with $1$ and $v$ a valuation on $R.$ Then $v$ is said to be \textbf{special*}, if $\textrm{Quot}(Rv) = K\nu.$
\end{definition} 

\noindent
Note that we write special*, because in \cite{kneb} ``special valuations'' refer to a different class of valuations. Let us give a few examples for special*-valuations.

\begin{lemma} Any Manis valuation is special*.
\end{lemma}
\begin{proof} By Lemma \ref{fields}, it suffices to show that $K\nu$ is a subfield of $\textrm{Quot}(Rv).$ We argue again via the homomorphism theorem. This time, we consider the map
\[
\phi: K_{\nu} \to \textrm{Quot}(Rv), \; \frac{x+\mathfrak{p}}{y+\mathfrak{p}} \mapsto \frac{x+I_v}{y+I_v}.
\]
Since $\mathfrak{p} \subseteq I_v,$ the choice of representatives in the domain does not matter. To show that $\phi$ is well defined, we additionally have to prove that $\phi$ maps any element $\frac{x+\mathfrak{p}}{y+\mathfrak{p}} \in K_v$ to an element in $\textrm{Quot}(Rv).$ The fact that this element lies in $K_v$ implies that $v'(y+\mathfrak{p}) \leq v'(x+\mathfrak{p}),$ and therewith $v(y) \leq v(x).$ Since $v$ is Manis, we find some $a \in R$ such that $v(a) = -v(y),$ i.e. $v(ay) = 0 \leq v(ax).$ Thus, $ay \in U_v$ and $ax \in R_v,$ so $\frac{ax+I_v}{ay+I_v} \in \textrm{Quot}(Rv).$ Therefore, also $\frac{x+I_v}{y+I_v} \in \textrm{Quot}(Rv),$ as this fraction equals $\frac{ax+I_v}{ay+I_v}.$ \vspace{1mm}

\noindent
It is easy to see that $\phi$ is a ring homomorphism. It remains to show that $I_{\nu}$ is the kernel of $\phi.$ Note that
\[
\frac{x+\mathfrak{p}}{y+\mathfrak{p}} \in \textrm{ker}(\phi) \Leftrightarrow v(x) > v(y) \Leftrightarrow v'(x+\mathfrak{p}) > v'(y+\mathfrak{p}) \Leftrightarrow \frac{x+\mathfrak{p}}{y+\mathfrak{p}} \in I_{\nu}.
\]
Applying the homomorphism theorem finishes the proof.
\end{proof}

\noindent
We give further examples to show that the class of special* valuations is strictly contained in the class of Manis valuations.

\begin{example} \label{special*} \hspace{7cm}
\begin{enumerate}
\item For any prime number $p \in \mathbb{N},$ the $p$-adic valuation $v$ on $R = \mathbb{Z}$ is special*. We have $$\textrm{Quot}(Rv) = \mathbb{F}_p = K\nu.$$ 
\vspace{1mm}

\item Let $S$ be a ring and $R = S[X].$ The degree valuation $v: R \to \mathbb{Z} \cup \{\infty\}, \; f \mapsto -\textrm{deg}(f)$ is special*. We have 
\[
\textrm{Quot}(Rv) = \mathbb{Q} = K\nu.
\] 
\end{enumerate}
\end{example}

\noindent
Note that for special* valuations there is a bijective correspondence between the orderings on $Rv$ and $K\nu = \textrm{Quot}(Rv).$ This allows us to reformulate the Baer-Krull Theorem for this class. Let $\mu$ denote the map sending $(\eta_{\trianglelefteq}, \trianglelefteq')$ to $(\eta_{\preceq},\preceq'),$ i.e. $\mu$ restricts the quasi-order $\trianglelefteq'$ on $Kv$ to the subring $Rv,$ while it is the identity in the first component, as, by the choice of the $\pi_i,$ the equality $\eta_{\preceq} = \eta_{\trianglelefteq}$ holds.

\begin{theorem} \label{BaerKrullIV} (Baer-Krull Theorem for quasi-ordered rings IV) \\
Let $R$ be a commutative ring with $1$ and $v$ a special* valuation on $R.$ Then the map
\begin{align*}
\mu \circ \psi \circ \lambda \colon \{\preceq \; \in \mathcal{X}_{\mathfrak{p}}(R)\colon \preceq \textrm{is } v\textrm{-compatible}&\} \to \mathcal{A}, \\ &\; \preceq \; \mapsto \; (\eta,\preceq')
\end{align*}
is a bijection, where $\mathcal{A} := \{-1,1\}^I \times \mathcal{X}_{o}(Rv) \sqcup \{1\}^I \times \mathcal{X}_{p}(Rv).$
\end{theorem} 

\begin{remark} \label{2div} In any of our four versions, the Baer-Krull Theorem simplies much further, if the value group $\Gamma_v$ is $2$-divisible, because then $\overline{\Gamma_v} = \Gamma_v/2\Gamma_v$ is trivial, and therefore $I = \emptyset.$ \vspace{1mm}

\noindent
So for instance in Theorem \ref{BaerKrullIV}, if $v$ is special* and $\Gamma_v$ is $2$-divisible, then there is a bijective correspondence

\begin{align*}
\mu \circ \psi \circ \lambda \colon \{\preceq \; \in \mathcal{X}_{\mathfrak{p}}(R)\colon \preceq \textrm{is } v\textrm{-compatible}&\} \to \mathcal{X}_{o}(Rv) \sqcup \mathcal{X}_{p}(Rv), \\ &\; \preceq \; \mapsto \; (\eta,\preceq').
\end{align*}
 
\end{remark}

\noindent
We conclude this paper by deducing Baer-Krull Theorems for ordered, respectively proper quasi-ordered, rings, from Theorem \ref{BaerKrullIV} and Corollary \ref{BaerKrullII}. The former statement immediately implies:

\begin{corollary} \label{BaerKrullorings} $\mathrm{(}$Baer-Krull Theorem for ordered rings $\mathrm{)}$ \\ 
Let $R$ be a commutative ring with $1$ and $v$ a special* valuation on $R.$ Then the map
\begin{align*}
\psi \colon \{\leq \in \mathcal{X}_{\mathrm{o},\mathfrak{q}_v}(R)\colon \leq \textrm{is } v\textrm{-compatible}&\} \to \{-1,1\}^I \times \mathcal{X}_{\mathrm{o},\{0\}}(Rv), \\ &\; \leq \; \mapsto \; (\eta_{\leq},\leq')
\end{align*}
is a bijection.
\end{corollary}

\noindent
If $R$ is a field, then this result coincides with Theorem \ref{baerkrullfield}. Further note that if $\Gamma_v$ is $2$-divisible, then Corollary \ref{BaerKrullorings} simplifies in the same manner as explained in Remark \ref{2div}. Moreover, the statement becomes evidently much easier if the domain $Rv$ is uniquely ordered.

\begin{lemma} 
For a domain $R,$ the following are equivalent:
\begin{itemize}
\item[(1)] $R$ is uniquely ordered. \vspace{1mm}

\item[(2)] $0$ is not a sum of non-zero squares and for each $a \in R,$ there exists some non-zero $b$ such that either $ab^2$ or $-ab^2$ is a sum of squares.
\end{itemize}
\begin{proof} In the proof we exploit the fact that $R$ is uniquely ordered if and only if $K:=\mathrm{Quot}(R)$ is uniquely ordered. Note that the latter is equivalent to the fact that for any $a \in K^*,$ either $a$ or $-a$ (and not both) is a sum of squares. \vspace{1mm}

\noindent
We first show that (2) implies (1). So let $\frac{x}{y} \in K^*$ with $x,y \in R.$ Then $xy \in R.$ So there exists some $0 \neq b$ such that (wlog) $xyb^2$ is a sum of squares in $R,$ say $xyb^2 = \sum p_i^2,$ with $p_i \in R.$ Then $$\frac{x}{y} = \sum_i \left(\frac{p_i}{yb}\right)^2$$ is a sum of squares in $K.$ Moreover, 
$-\frac{x}{y}$ is not a sum of squares in $K$, since otherwise $0$ would be a sum of non-zero squares in $R.$ \vspace{1mm}

\noindent 
We conclude by showing that (1) implies (2). So suppose that $R$ is uniquely ordered, i.e. also $K$ is uniquely ordered. Hence, $0$ is not a sum of non-zero squares in $K,$ but then this is also the case in $R.$ Now let $a \in R \subseteq K.$ Then $a$ or $-a$ is a sum of squares, say
$$ \pm a = \sum_i \left(\frac{x_i}{y_i}\right)^2 \quad (x_i,y_i \in R).$$ This yields $\pm a\prod_iy_i^2 = \sum_i \left(x_i\prod_{j\neq i}y_j\right)^2.$ Hence, $b:= (\prod_i y_i)^2$ satisfies (2). 
\end{proof} 
\end{lemma}

\noindent
Our version of the Baer-Krull Theorem allows us to transfer \cite[Corollary 2.2.6]{prestel} to the ring case. In analogy to the field case, we call an ordered ring $(R,\leq)$ Archimedean, if for any $x \in R$ there exists some $n \in \N$ such that $x < n,$ and otherwise non-Archimedean.

\begin{corollary} \hspace{7cm}
\begin{enumerate}
 \item If $R$ carries a non-trivial Manis valuation whose residue class domain is real, then $R$ admits a non-Archimedean ordering. \vspace{1mm}

 \item Conversely, if $R$ carries a non-Archimedean ordering, then $R$ admits a non-trivial valuation with real residue class domain.
\end{enumerate}
\begin{proof} \hspace{7cm}
\begin{enumerate}
 \item The non-Archimedean ordering is derived exactly as in the field case. So let $v$ be a non-trivial Manis valuation on $R,$ and let $\overline{\leq}$ denote an ordering on $Rv.$ Choosing $\eta = 1$ and applying the Baer-Krull Theorem for ordered rings yields an ordering $\leq$ on $R$ such that $v$ and $\leq$ are compatible. Particularly, $R_v$ is a convex subring of $R$ by Theorem \ref{qoringcomp}. Moreover, $R_v \subsetneq R,$ since $v$ is a non-trivial Manis valuation. Now let 
$Z := Z(\leq)$ denote the convex hull of $\Z$ in $R.$ Note that $Z$ is the smallest convex subring of $R,$ so $Z \subseteq R_v \subsetneq R.$ This yields that $\leq$ is a non-Archimedean ordering on $R.$ \vspace{1mm}

 \item For the converse we first consider the case where $R$ is a domain. Let $\leq$ denote a non-Archimedean ordering on $R.$ Then $\leq$ uniquely extends to a non-Archimedean ordering on $K:=\mathrm{Quot}(R).$ From \cite[Corollary 2.2.6]{prestel} follows that $K$ carries a non-trivial valuation $w$ with real residue class field $Kw.$ Note that the restriction $v$ of $w$ to $R$ is a non-trivial (not necessarily Manis) valuation on $R.$ Moreover, the map 
$\phi: R_v \to Kw, x \mapsto x+I_w$ is a ring homomorphism with kernel $I_v,$ so $Rv$ is real as a subring of the real ring $Kw.$ \vspace{1mm}

\noindent
For the general case, note that if $R$ carries a non-Archimedean ordering $\leq,$ then $\overline{x} \leq' \overline{y} :\Leftrightarrow x \leq y$ defines a non-Archimedean ordering on the domain $R/E_0$ (see \cite[Lemma 4.1]{sim}). Hence, there exists a non-trivial valuation $w$ on $R/E_0$ such that $(R/E_0)w$ is real. As was shown in \cite[Lemma 4.4]{sim}, this yields a valuation $v$ on $R$ with support $E_0$ via $v(x) = w(\overline{x}),$ and the value groups of $v$ and $w$ coincide, i.e. $v$ is non-trivial as well. By definition of $v,$ it is easy to see that $Rv$ inherits the order from $(R/E_0)w.$
\end{enumerate}

\end{proof}
\end{corollary}

\begin{remark} In the first statement of the previous corollary, the assumption that the valuation is Manis is crucial in order to obtain that $R_v \subsetneq R.$ Note that there are non-trivial special* valuation such that $R_v = R,$ for instance any $p$-adic valuation on $\mathbb{Z}.$ \vspace{1mm}

\noindent 
For the converse, we can not derive surjectivity of $v,$ because the restriction of a field valuation to a subring is in general not Manis. For instance any field valuation restricted to the integers is either trivial or not Manis. since $\Z$ admits no non-trivial Manis valuation. The latter is due to the fact that the triangle inequality yields $v(n) \geq 0$ for any natural number $n.$ 
\end{remark}

\noindent
The Baer-Krull Theorem for quasi-ordered rings also gives rise to a characterization of all Manis valuations $w$ on $R,$ that are finer than $v,$ if we additionally assume that $v$ is non-trivial.

\begin{corollary} \label{BaerKrullqoringsI} $\mathrm{(}$Baer-Krull Theorem for proper quasi-ordered rings I $\mathrm{)}$ \\ 
Let $R$ be a commutative ring with $1$ and $v$ a special* valuation on $R.$ Then the map
\begin{align*}
\psi \colon \{\preceq_w \in \mathcal{X}_{\mathrm{p},\mathfrak{q}_v}(R)\colon \preceq_w \textrm{is } v\textrm{-compatible}&\} \to \mathcal{X}_{\mathrm{p},\{0\}}(Rv), \\ &\; \preceq \; \mapsto \; \preceq'
\end{align*}
is a bijection.

\end{corollary}

\noindent
Now recall from Theorem \ref{qoringcomp} and Remark \ref{qoringcompapp}(4) that if $\preceq = \preceq_w$ is $v$-compatible, then $\preceq' = \preceq_{w/v}$ (see Remark \ref{qoringcompapp}(4) for the proof and a definition of $w/v$. Further note that the Manis property is not required for deducing (4) from (1) in Theorem \ref{qoringcomp}). This allows us to reformulate the previous corollary more precisely (see Corollary \ref{BaerKrullqoringsII}).

\begin{lemma} \label{qoringcompapp2} Let $(R,v)$ be a valued ring for some Manis valuation $v$ on $R,$ and let $w$ be a valuation on $R$ such that $\preceq_w$ is $v$-compatible and $\mathfrak{q}_v = \mathfrak{q}_w.$ Then $w$ is 
Manis if and only if $w/v$ is Manis.
\begin{proof} If $u$ is some arbitrary valuation of $R,$ then $u(R \backslash \mathfrak{q}_u)$ is additively closed by axiom (V3) of Definition \ref{ringval}. So in order to show that $u$ is Manis, it suffices to prove that $u(R \backslash \mathfrak{q}_u)$ is closed under additive inverses. \vspace{1mm}

\noindent
Suppose that $w$ is Manis. Let $\gamma := w/v(\overline{a}) \in \Gamma_{w/v}$ be arbitrary, $a \in U_v.$ Then $w/v(\overline{a}) = w(a).$ Since $w$ is Manis, there exists some $b \in R$ such that $w(b) = -w(a).$ Thus, $w(ab) = 0 = w(1).$ By $v$-compatibility of $\preceq_w,$ we obtain that also $v(ab) = 0.$ Since $a \in U_v,$ also $b \in U_v.$ Therefore, it holds $w/v(\overline{b}) = w(b) = -\gamma \in \Gamma_{w/v}.$ \vspace{1mm}

\noindent
Now assume that $w/v$ is Manis, and let $a \in R$ such that $w(a) =: \gamma \in \Gamma_w.$ We show that there exists some $b \in R$ with $w(b) = -\gamma.$ Note that $a \notin \mathfrak{q}_v,$ since $\mathfrak{q}_v = \mathfrak{q}_w.$ Since $v$ is Manis, we find some $y \in R$ such that $ay \in U_v.$ So 
$w/v(\overline{ay}) = w(ay) =: \gamma_1.$ By surjectivity of $w/v,$ there exists some $z \in R$ such that $w/v(\overline{z}) = w(z) = -\gamma_1.$ Therefore, 
$w(z) = -w(a)-w(y).$ This yields $w(yz) = -w(a) = -\gamma,$ i.e. $b = yz.$ 
\end{proof}

\end{lemma}

\begin{corollary} \label{BaerKrullqoringsII} $\mathrm{(}$Baer-Krull Theorem for proper quasi-ordered rings II $\mathrm{)}$ \\ 
Let $R$ be a commutative ring with $1$ and $v$ a Manis valuation on $R.$ Then the map
\begin{align*}
\psi \colon \{w \colon w \textrm{ Manis}, \, \preceq_w v\textrm{-comp.}, \, \mathfrak{q}_w = \mathfrak{q}_v &\} \to \{u \colon u \textrm{ Manis val. on } Rv, \, \mathfrak{q}_{u} = \{0\}\}, \\ &\, w \mapsto w/v
\end{align*}
is a bijection.
\begin{proof} We deduce this corollary from Corollary \ref{BaerKrullqoringsI}. As mentioned above, if $\preceq = \preceq_w$ is a proper quasi-order compatible with $v,$ then $\preceq' = \preceq_{w/v}.$ Moreover we have shown in the previous lemma that $w$ is Manis if and only if $w/v$ is Manis. So we may restrict both the domain and co-domain of $\psi$ to proper quasi-orders that come from a Manis valuation. 
\end{proof}
\end{corollary}

\noindent
Since $v$ and $w$ are both Manis and $\preceq_w$ is compatible with $v,$ it follows via Lemma \ref{coarse} that the previous corollary characterizes precisely all Manis refinements $w$ of $v,$ if the valuation $v$ 
(and then also $w$) is non-trivial.

\vspace{4mm}

\textsc{Fachbereich Mathematik und Statistik, Universit\"at Konstanz},

78457 \textsc{Konstanz, Germany}, 

E-mail address: salma.kuhlmann@uni-konstanz.de 

E-mail address: simon.2.mueller@uni-konstanz.de
\end{document}